\documentclass[10pt,oneside,english]{amsart}

\usepackage{amssymb}

\makeatletter
 \theoremstyle{plain}
\newtheorem{thm}{Theorem}[section]
  \theoremstyle{definition}
  \newtheorem{defn}[thm]{Definition}
  \theoremstyle{plain}
  \newtheorem{lem}[thm]{Lemma}
  \theoremstyle{remark}
  \newtheorem*{rem*}{Remark}
  \theoremstyle{plain}
  \newtheorem{cor}[thm]{Corollary}
  \theoremstyle{plain}

\usepackage{babel}
\makeatother

\begin{document}

\title{TRAVELING WAVE SOLUTIONS FOR LOTKA-VOLTERRA SYSTEM RE-VISITED}

\maketitle
\renewcommand{\theequation}{\arabic{section}.\arabic{equation}}

\begin{center}
ANTHONY W LEUNG$^{a,}$%
\footnote{Contacting author

\medskip{}


}, XIAOJIE HOU$^{b}$, \textsc{WEI FENG$^{b}$}
\par\end{center}

\medskip{}

{\footnotesize \centerline{$^{a}$Department of Mathematical Sciences,}
\centerline{ University of Cincinnati,} \centerline{ Cincinnati,
OH 45221 }}{\footnotesize \par}

\smallskip{}

{\footnotesize \centerline{$^{b}$Department of Mathematics and Statistics,}
\centerline{ University of North Carolina at Wilmington,} \centerline{
Wilmington, NC 28403 }}{\footnotesize \par}

\begin{abstract}
Using a new method of monotone iteration of a pair of smooth lower-
and upper-solutions, the traveling wave solutions of the classical
Lotka-Volterra system are shown to exist for a family of wave speeds.
Such constructed upper and lower solution pair enables us to derive
the explicit value of the minimal (critical) wave speed as well as
the asymptotic rates of the wave solutions at infinities. Furthermore,
the traveling wave corresponding to each wave speed is unique modulo
a translation of the origin. The stability of the traveling wave solutions
with non-critical wave speed is also studied by spectral analysis
of the linearized operator in exponentially weighted Banach spaces.
\end{abstract}
\medskip{}

\section{\textbf{INTRODUCTION }}

\setcounter{equation}{0} 

We re-visit the classical Lotka-Volterra competition system:

\begin{equation}
\left\{ \begin{array}{l}
u_{t}=d_{1}u_{xx}+u(a_{1}-b_{1}u-c_{1}v),\\
\\v_{t}=d_{2}v_{xx}+v(a_{2}-b_{2}u-c_{2}v)\end{array}\quad\quad(x,t)\in\mathbb{R\times\mathbb{R}}^{+},\right.\label{eq:1.1}\end{equation}
where $u(x,t)$, $v(x,t)$ are the population densities of  two competing
species, the constants $d_{i}$, $a_{i}$, $b_{i}$, $c_{i}$, $i=1,2$
are assumed to be positive. In this paper, we are trying to acomplish
the following goals: providing a new and easy construction of upper-
and lower-solutions to derive the traveling wave solutions of \eqref{eq:1.1};
obtaining an accurate description of the minimal wave speed and asymptotic
behaviors (up to the first order) of the wave solutions; investigating
the stability of the traveling wave solutions in various Banach spaces.

System \eqref{eq:1.1} has been extensively studied. In \cite{Leung}
and \cite{Pao}, there are many applications and treatments of solutions
of \eqref{eq:1.1} in bounded spatial domain under various initial
and boundary conditions. As is well known, system \eqref{eq:1.1}
and its cooperative counter-parts admit traveling wave solutions.
In \cite{TangFife}, Tang and Fife showed the existence of the traveling
wave solutions connecting the extinction state with the co-existent
state, In \cite{KanelZhou}, Kanel and Zhou studied the existence
of the traveling wave solutions connecting the coexistent state to
a semi-exitinction state. In \cite{Fei}, the traveling wave solution
connectiong two semi-extinction states were studied, they also estimated
the minimal wave speed. For the other treatment of the traveling wave
solutions of system \eqref{eq:1.1} and its generalizations, please
see \cite{WuZou}, \cite{AlexanderGardnerJones}, \cite{Hosono},
\cite{Volpert}.

Throughout this paper, we make the following assumptions:

\begin{itemize}
\item {[}H1]. ${\displaystyle \frac{a_{2}}{b_{2}}<\frac{a_{1}}{b_{1}},}$ 
\item {[}H2]. ${\displaystyle \frac{a_{2}}{c_{2}}<\frac{a_{1}}{c_{1}}}$,
\item {[}H3]. ${\displaystyle \frac{b_{2}}{b_{1}}+\frac{c_{1}a_{2}}{c_{2}a_{1}}\le1+\frac{a_{2}}{a_{1}}.}$
\end{itemize}
Under these conditions system \eqref{eq:1.1} has three equilibria
$(0,0)$, $(\frac{a_{1}}{b_{1}},0)$ and $(0,\frac{a_{2}}{b_{2}}$).

We will use a new monotone iteration method to investigate the traveling
wave solutions of \eqref{eq:1.1}. The traveling wave solution connects
two of the above equilibria. The monotone iteration method has been
widely used in the study of the traveling wave solutions of reaction
diffusion system such as \eqref{eq:1.1}, but most constructed upper-
and lower-solutions in literature are non-smooth. The relatively larger 'gap' between the non-smooth upper and lower-solutions
creates certain difficulties in deriving the accurate asymptotic estimates
of the traveling wave solutions at infinities. Such estimates are
valuable in applications, and enable one fully exploit the cooperative or competitive structure of the Lotka Volterra system.

The smooth upper-solution in the monotone iteration as in section
3 is derived from the traveling wave solutions of the KPP equation.
Observing that if we take fuction $v$ to be a constant, then the
first equation of \eqref{eq:1.1} is a generalized  KPP equation, the same consideration
is also true for the second equation. The existence, uniqueness, asymptotics
as well as the stability of the traveling wave solutions of the KPP
system are well known, so are the properties of the upper-solution. The most difficult part in section 3 is to
construct the lower solution for system \eqref{eq:1.1}. Since such
constructed upper solution is 'nearly' a solution, a compromise is
made to derive the smooth lower-solution. In fact, the lower-solution
does not satisfy the boundary condition at $\infty$. Thanks to the
realxed condition in \cite{WuZou}, we can still apply the monotone
iteration scheme as specified in \cite{WuZou}, \cite{BoumenirNguyen}.
The trade off of such 'shorter' lower solution is that we can have
some freedom to choose the lower solution with desired asymptotic
rate at negative infinity. This leads to an accurate (up to first order)
asymptotic estimates of the traveling wave solution at $-\infty$,
and an exact value of the minimal wave speed.

The asymptotics of the traveling wave solutions at infinities are
obtained by comparison principle. Once the asymptotic behaviors of
the traveling wave solutions are known, we can use the Maximum principle
and Sliding domain method \cite{LeungHouLi}to derive the uniqueness,
strict monotonicity as well as the local stability of the traveling
wave solutions.

The local stability of the traveling wave solutions is studied by
means of spectral analysis in some weighted Banach spaces. We proceed
to show that the linearized operator about the traveling wave solution
has essential spectrum lying completely in the left complex plane,
and that $0$ is not an eigenvalue of the linearized operator in the
weighted Banach spaces, whereas all the other eigenvalues of the linearized
operator have negative real parts.
This means the traveling wave solution is linearly exponentially stable.
Since such linear operator is sectorial, the linear stability implies
the local nonlinear stability of the traveling wave solutions \cite{AlexanderGardnerJones},
\cite{Volpert}, \cite{Henry}. Though general theories on the stability
of the traveling wave solutions are known \cite{Volpert}, \cite{Kapitula},
\cite{AlexanderGardnerJones}, the verification of the conditions
there is in fact a case by case study. The methods used in the stability
study of the traveling wave solutions of \eqref{eq:1.1} are similar
in the spirit to those in \cite{XuZhao}, \cite{Sattinger}, \cite{Kapitula},
\cite{AlexanderGardnerJones}, \cite{WuXingYe}, \cite{WuLi}, \cite{BatesChen},
but in comparing to the above methods, we need to further overcome
the 'unstable' component of the system. This is done by studying an
equivalent form of the linearized operator in a  smaller weighted
Banach spaces. 

We remark that the stability is only for the traveling waves with
non-critical wave speed. The stability of the traveling waves with
critical wave speed is currently under investigation.

The paper is arranged as follows: in section 2, we study the steady
states of the system \eqref{eq:1.1} and obtain the attraction region
of the stable steady state; in section 3, we show there are traveling
wave solutions connecting the one of the unstable steady states with
a stable one, corresponding to each wave speed, the traveling wave
solution is unique and strictly monotone. The analysis is done by utilizing 
an accurate description of asymptotic behavior of the traveling wave
solutions. Furthermore, we have obtained the estimate of the critical
wave speed. In the last section of the paper, we show the traveling
wave solutions are locally, nonlinearly exponentially stable.

\section{\textbf{THE EQUILIBRIA AND THEIR STABILITY}\label{sec:2}\setcounter{equation}{0}}

In this section, we analyze the constant steady states of system (\ref{eq:1.1})
under conditions H1-H3. 

Consider system \eqref{eq:1.1} with the initial conditions \begin{equation}
u(x,0)=u_{0}(x),\quad v(x,0)=v_{0}(x)\quad x\in\mathbb{R}.\label{eq:2.1}\end{equation}
where $u_{0}$ and $v_{0}$ are non-negative bounded smooth functions
on $\mathbb{R}$. 

\begin{defn}
\label{def:2.1}A pair of nonnegative bounded smooth functions $\tilde{U}=(\tilde{u},\tilde{v})$
and $\hat{U}=(\hat{u},\hat{v})$ are called coupled upper- and lower-solutions
of the Cauchy problem \eqref{eq:1.1}-\eqref{eq:2.1} if $\tilde{U}\geq\hat{U}$
on $\mathbb{R\times}[0,\infty)$ and the following inequalities are
satisfied\begin{equation}
\begin{array}{ccl}
{\displaystyle \tilde{u}_{t}} & \geq & d_{1}\tilde{u}_{xx}+\tilde{u}(a_{1}-b_{1}\tilde{u}-c_{1}\hat{v}),\\
\\{\displaystyle \tilde{v}_{t}} & \geq & d_{2}\tilde{v}_{xx}+\tilde{v}(a_{2}-b_{2}\hat{u}-c_{2}\tilde{v});\\
\\{\displaystyle \hat{u}_{t}} & \leq & d_{1}\hat{u}_{xx}+\hat{u}(a_{1}-b_{1}\hat{u}-c_{1}\tilde{v}),\\
\\{\displaystyle \hat{v}_{t}} & \leq & d_{2}\hat{v}_{xx}+\hat{v}(a_{2}-b_{2}\tilde{u}-c_{2}\hat{v}),\end{array}\quad\mbox{in}\;\mathbb{R}\times(0,\infty),\label{eq:2.2}\end{equation}
 as well as the initial conditions \begin{equation}
\tilde{u}(x,0)\geq u_{0}(x)\geq\hat{u}(x,0),\tilde{v}(x,0)\geq v_{0}(x)\geq\hat{v}(x,0)\quad\mbox{in}\quad\mathbb{R}.\label{eq:2.3}\end{equation}

\end{defn}
It is known from \cite{Leung}, \cite{Pao} that if there exist coupled
upper and lower solutions $\tilde{U}$ and $\hat{U}$ on $\mathbb{R}\times[0,\infty)$,
then the Cauchy problem \eqref{eq:1.1}-\eqref{eq:2.1} has a unique
solution $U(x,t)=(u(x,t),v(x,t))$ with $\tilde{u}(x,t)\geq u(x,t)\geq\hat{u}(x,t)$
and $\tilde{v}(x,t)\geq v(x,t)\geq\hat{v}(x,t)$ on $\mathbb{R}\times[0,\infty)$.
The next theorem gives the asymptotic stability and attraction region
of the equilibrium $(\frac{a_{1}}{b_{1}},0)$.

\begin{thm}
\label{thm2.2}Let $\alpha=b_{2}(\frac{a_{1}}{b_{1}}-\frac{a_{2}}{b_{2}})$,
$B=\frac{b_{1}b_{2}}{a_{1}c_{1}}(\frac{a_{1}+a_{2}}{b_{2}}-\frac{a_{1}}{b_{1}})$,
and $\beta==\frac{b_{1}b_{2}}{a_{1}}(\frac{a_{1}}{b_{1}}-\frac{a_{2}}{b_{2}})$.
Assuming that the hypotheses (H1) and (H3) hold, we know that $\alpha>0$,
$B\geq0$, and $\beta>0$. If for some $A>0$ and $0<\rho(0)<a_{1}/b_{1}$,
the initial functions satisfy \[
(\frac{a_{1}}{b_{1}}-\rho(0),0)\leq(u_{0}(x),v_{0}(x))\leq(\frac{a_{1}}{b_{1}}+A\rho(0),B\rho(0))\]
 on $R$, then the solution for \eqref{eq:1.1}-\eqref{eq:2.1} satisfies
\begin{equation}
(\frac{a_{1}}{b_{1}}-\rho(t),0)\leq(u(t,x),v(t,x))\leq(\frac{a_{1}}{b_{1}}+A\rho(t),B\rho(t))\label{eq:2.4}\end{equation}
 on $[0,\infty)\times R$ where \begin{equation}
\rho(t)=[\beta/\alpha+(\rho(0)^{-1}-\beta/\alpha)e^{\alpha t}]^{-1}.\label{eq:2.5}\end{equation}
 
\end{thm}
\begin{proof}
We will show that $(\tilde{u},\tilde{v})=(\frac{a_{1}}{b_{1}}+A\rho(t),B\rho(t))$
and $(\hat{u},\hat{v})=(\frac{a_{1}}{b_{1}}-\rho(t),0)$ are a pair
of coupled upper-lower solutions defined in Definition \eqref{def:2.1}.
One can easily see that $\hat{v}=0$ satisfies the required inequalities
in \eqref{eq:2.2}. 

We first start with $\tilde{u}=\frac{a_{1}}{b_{1}}+A\rho(t)$. From
\eqref{eq:2.2}, it needs to satisfy the differential inequality \[
A\rho'(t)\geq(\frac{a_{1}}{b_{1}}+A\rho)[a_{1}-b_{1}(\frac{a_{1}}{b_{1}}+A\rho)]=-A(a_{1}\rho+b_{1}A\rho^{2})\]
 For $\rho(t)\geq0$, it suffices to show that \begin{equation}
\begin{array}{ccl}
\rho'(t)+a_{1}\rho(t) & \geq & -b_{1}A\rho^{2}(t).\end{array}\label{eq:2.6}\end{equation}
Also for $\tilde{v}=B\rho(t)$, we see from \eqref{eq:2.2} that it
needs to satisfy the differential inequality \[
B\rho'(t)\geq B\rho[a_{2}-b_{2}(\frac{a_{1}}{b_{1}}-\rho)-c_{2}B\rho]\]
 and it suffices to show that \begin{equation}
\begin{array}{ccl}
\rho'(t)+b_{2}(\frac{a_{1}}{b_{1}}-\frac{a_{2}}{b_{2}})\rho(t) & \geq & (b_{2}-c_{2}B)\rho^{2}(t).\end{array}\label{eq:2.7}\end{equation}
Finally we look at $\hat{u}=\frac{a_{1}}{b_{1}}-\rho(t)$. Again from
\eqref{eq:2.2}, it needs to satisfy the differential inequality \[
-\rho'(t)\leq(\frac{a_{1}}{b_{1}}-\rho)[a_{1}-b_{1}(\frac{a_{1}}{b_{1}}-\rho)-c_{1}B\rho],\]
 or equivalently, \[
\rho'(t)\geq(\frac{a_{1}}{b_{1}}-\rho)(c_{1}B-b_{1})\rho.\]
 For $\rho(t)\geq0$ it suffices to show that \begin{equation}
\begin{array}{ccl}
\rho'(t)+\frac{a_{1}}{b_{1}}(b_{1}-c_{1}B)\rho(t) & \geq & (b_{1}-c_{1}B)\rho^{2}(t).\end{array}\label{eq:2.8}\end{equation}
From hypotheses (H1) and (H3) we observe the fact that \[
a_{1}>\frac{a_{1}b_{2}}{b_{1}}-a_{2}=b_{2}(\frac{a_{1}}{b_{1}}-\frac{a_{2}}{b_{2}})=\alpha>0\]
 Setting $\frac{a_{1}}{b_{1}}(b_{1}-c_{1}B)=\alpha$ leads to the
choice of 

\begin{equation}
B=\frac{b_{1}b_{2}}{a_{1}c_{1}}(\frac{a_{1}+a_{2}}{b_{2}}-\frac{a_{1}}{b_{1}})>0.\label{eq:2.9}\end{equation}
From the hypotheses (H3), one can obtain that \[
(b_{1}-c_{1}B)-(b_{2}-c_{2}B)=(b_{2}-\frac{a_{2}b_{1}}{a_{1}})-(b_{2}-c_{2}B)\]
 \[
=\frac{b_{1}c_{2}}{a_{1}c_{1}}(a_{1}+a_{2})-\frac{b_{2}c_{2}}{c_{1}}-\frac{a_{2}b_{1}}{a_{1}}\]
 \[
=\frac{b_{1}c_{2}}{c_{1}}[1+\frac{a_{2}}{a_{1}}-\frac{b_{2}}{b_{1}}-\frac{a_{2}c_{1}}{a_{1}c_{2}}]\geq0.\]
Noting that $b_{1}-c_{1}B=\frac{b_{1}b_{2}}{a_{1}}(\frac{a_{1}}{b_{1}}-\frac{a_{2}}{b_{2}})=\beta$,
we can now conclude that the three differential inequalities \eqref{eq:2.6}-\eqref{eq:2.8}
will all be satisfied if the function $\rho(t)$ is a positive solution
of the differential equation \begin{equation}
\rho'(t)+\alpha\rho(t)=\beta\rho^{2}(t).\label{eq:2.10}\end{equation}
This leads to the function $\rho(t)$ given in \eqref{eq:2.5} with
$\rho(0)<\alpha/\beta=a_{1}/b_{1}$ and $\lim_{t\rightarrow\infty}\rho(t)=0$
. 
\end{proof}
From the arbitrariness of constant $A$ in Theorem \eqref{thm2.2},
we then have the the attraction region for the equilibrium $(a_{1}/b_{1},0)$.
When the hypotheses (H1) and (H3) hold, for all the initial density
functions $(u_{0},v_{0})$ in the rectangular area \[
(0,\infty)\times[0,\frac{b_{2}}{c_{1}}(\frac{a_{1}+a_{2}}{b_{2}}-\frac{a_{1}}{b_{1}})),\]
 the solution $(u,v)$ of the system \eqref{eq:1.1}-\eqref{eq:2.1}
converges to the equilibrium $(a_{1}/b_{1},0)$ uniformly on $\mathbb{R}$
as $t\rightarrow\infty$ with the rate $e^{-\alpha t}$.

In the meantime, by adding hypothesis (H2), we can quickly find that
the equilibriums $(0,a_{2}/c_{2})$ and $(0,0)$ are both unstable.
For this purpose we construct a pair of upper-lower solutions \[
(\tilde{u},\tilde{v})=(M,a_{2}/c_{2}-\rho(t))~~\mbox{and}~~(\hat{u},\hat{v})=(A\rho(t),0),\]
 where $M\geq a_{1}/b_{1}$ is a constant. Constant $A$ and function
$\rho(t)$ will be determined later. The differential inequalities
in \eqref{eq:2.2} are automatically satisfied by $\tilde{u}$ and
$\hat{v}$. For $\hat{u}$ and $\tilde{v}$, the following relations
need to hold: \[
\rho'(t)\leq(a_{1}-\frac{a_{2}c_{1}}{c_{2}})\rho+(c_{1}-Ab_{1})\rho^{2},\]
and \[
\rho'(t)\leq(\frac{a_{2}}{c_{2}}-\rho)(c_{2}-Ab_{2})\rho.\]

The above inequalities are equivalent to

\begin{equation}
\begin{array}{c}
\rho'(t)-(a_{1}-\frac{a_{2}c_{1}}{c_{2}})\rho\leq-(b_{1}A-c_{1})\rho^{2},\\
\\\rho'(t)-\frac{a_{2}}{c_{2}}(b_{2}A-c_{2})\rho\leq-(b_{2}A-c_{2})\rho^{2}.\end{array}\label{eq:2.11}\end{equation}

Setting \[
a_{1}-\frac{a_{2}c_{1}}{c_{2}}=\frac{a_{2}}{c_{2}}(b_{2}A-c_{2}),\]
 from hypothesis (H2) we have \begin{equation}
A=\frac{1}{b_{2}}(c_{2}-c_{1}+\frac{a_{1}c_{2}}{a_{2}})>\frac{c_{2}}{b_{2}}.\label{eq:2.12}\end{equation}

From the hypotheses (H3), one can obtain that \[
(b_{1}A-c_{1})-(b_{2}A-c_{2})=\frac{b_{1}}{b_{2}}(c_{2}-c_{1}+\frac{a_{1}c_{2}}{a_{2}})-\frac{a_{1}c_{2}}{a_{2}}\]
 \[
=\frac{b_{1}c_{2}}{a_{1}c_{1}}(a_{1}+a_{2})-\frac{b_{2}c_{2}}{c_{1}}-\frac{a_{2}b_{1}}{a_{1}}\]
 \[
=\frac{a_{1}b_{1}c_{2}}{a_{2}b_{2}}[1+\frac{a_{2}}{a_{1}}-\frac{b_{2}}{b_{1}}-\frac{a_{2}c_{1}}{a_{1}c_{2}}]\geq0.\]

Both the inequalities in \eqref{eq:2.11} can be satisfied by choosing
$\rho(t)$ as the solution of the differential equation \begin{equation}
\rho'(t)-\gamma\rho=-\delta\rho^{2}\label{eq:2.13}\end{equation}
 where \[
\gamma=a_{1}-\frac{a_{2}c_{1}}{c_{2}}>0~~\mbox{and}~~\delta=b_{1}A-c_{1}\geq b_{2}A-c_{2}=\frac{a_{1}c_{2}}{a_{2}}-c_{1}>0.\]

This results in the function \begin{equation}
\rho(t)=\frac{\gamma}{\delta+Ce^{-\gamma t}}\label{eq:2.14}\end{equation}
 with an arbitrary constant $C>0$. For arbitrarily small $\epsilon>0$,
one can always find a constsnt $C$ large enough such that $\rho(0)=\gamma/(\delta+C)<\epsilon$.
The fact that $\lim_{t\rightarrow\infty}\rho(t)=\gamma/\delta$ leads
to the following theorem indicating that $(0,a_{2}/c_{2})$ and $(0,0)$
are both unstable.

\begin{thm}
Let \begin{equation}
\gamma=a_{1}-\frac{a_{2}c_{1}}{c_{2}},~A=\frac{1}{b_{2}}(c_{2}-c_{1}+\frac{a_{1}c_{2}}{a_{2}})\quad\mbox{and}\quad\delta=\frac{b_{1}}{b_{2}}(c_{2}-c_{1}+\frac{a_{1}c_{2}}{a_{2}})-c_{1}.\label{eq:2.15}\end{equation}
Assuming that the hypotheses (H1), (H2) and (H3) hold, we know that
$\gamma>0$, $A>c_{2}/b_{2}$ and $\delta\geq\frac{a_{1}c_{2}}{a_{2}}-c_{1}>0$.
For any arbitrarily small $\epsilon$ with $0<\epsilon<\min\{A\gamma/\delta,\gamma/\delta\}$,
if the initial functions $(u_{0},v_{0})$ satisfies $u_{0}(x)\geq\epsilon$
and $0\leq v_{0}(x)\leq a_{2}/c_{2}-\epsilon$ on $\mathbb{R}$, then
the solution $(u(x,t),v(x,t))$ of \eqref{eq:1.1} and \eqref{eq:2.1}
satisfies \begin{equation}
\liminf_{t\rightarrow\infty}u(x,t)\geq\frac{A\gamma}{\delta}\quad\mbox{and}\quad\limsup_{t\rightarrow\infty}v(x,t)\leq\frac{a_{2}}{c_{2}}-\frac{\gamma}{\delta}.\label{eq:2.16}\end{equation}

\end{thm}

\section{\textbf{\label{sec:3}THE TRAVELING WAVES}\setcounter{equation}{0}}

In section \ref{sec:2}, we showed system (\ref{eq:1.1}) has two
unstable constant steady states: $(0,0)$, $(0,\frac{a_{2}}{c_{2}})$
and one asymptotically stable constant steady state $(\frac{a_{1}}{b_{1}},0)$.
We will show that there are traveling wave solutions of (\ref{eq:1.1})
having the form

\begin{equation}
(u(x,t),v(x,t))=(kw(\sqrt{\frac{a_{1}}{d}}x+ca_{1}t),qz(\sqrt{\frac{a_{1}}{d}}x+ca_{1}t)),\label{eq:3.1}\end{equation}
and connecting the unstable state $(0,\frac{a_{2}}{c_{2}})$ with
$(\frac{a_{1}}{b_{1}},0)$ as the variable $\sqrt{\frac{a_{1}}{d}}x+ca_{1}t$
runs from $-\infty$ to $+\infty$. The constant $c$ in \eqref{eq:3.2}
is the wave speed and the minimal speed is also called the critical
wave speed. Throughout the rest of the paper, we assume $d_{1}=d_{2}=d$.

To simplify notions, we introduce the following transformations to
\eqref{eq:1.1}:

\begin{equation}
\begin{array}{l}
r=a_{1}^{-1}c_{1}q,\,\,\,\epsilon_{1}=a_{1}^{-1}a_{2},\\
\\b=b_{2}b_{1}^{-1},\,\,\,\epsilon_{2}=a_{2}^{-1}c_{2}q-1,\\
\\k=a_{1}b_{1}^{-1},\,\,\,\hbox{\mbox{and}}\,\,\, q\,\,\hbox{\mbox{is a constant satisfying}}\,\,\, a_{2}c_{2}^{-1}<q<a_{1}c_{1}^{-1}.\end{array}\label{eq:3.2}\end{equation}

Under transformations (\ref{eq:3.1}) and (\ref{eq:3.2}), system
(\ref{eq:1.1}) is changed into

\begin{equation}
\left\{ \begin{array}{l}
-w_{\xi\xi}+cw_{\xi}=w(1-w-rz)\\
\hspace{200pt}-\infty<\xi<\infty,\\
-z_{\xi\xi}+cz_{\xi}=z(\epsilon_{1}-bw-\epsilon_{1}(1+\epsilon_{2})z)\end{array}\right.\label{eq:3.3}\end{equation}
with the corresponding boundary conditions 

\begin{equation}
\left\{ \begin{array}{lll}
\lim_{\xi\rightarrow-\infty}(w(\xi),z(\xi)) & = & (0,{\displaystyle \frac{1}{1+\epsilon_{2}}),}\\
\\\lim_{\xi\rightarrow\infty}(w(\xi),z(\xi)) & = & (1,0).\end{array}\right.\label{eq:3.4}\end{equation}
where $\xi=\sqrt{\frac{a_{1}}{d}}x+ca_{1}t$ in \eqref{eq:3.3}, \eqref{eq:3.4}
for $x\in\mathbb{R}$ and $t\in\mathbb{R}^{+}$ .

We further introduce the transformations

\begin{equation}
u_{1}(\xi)=w(\xi),\quad u_{2}(\xi)={\displaystyle \frac{1}{1+\epsilon_{2}}-z(\xi)}\label{eq:3.5}\end{equation}
to change system (\ref{eq:3.3}) into the following monotone (cooperative)
system:

\begin{equation}
\left\{ \begin{array}{lll}
-(u_{1})_{\xi\xi}+c(u_{1})_{\xi} & = & u_{1}({\displaystyle \frac{1+\epsilon_{2}-r}{1+\epsilon_{2}}-u_{1}+ru_{2}),}\\
\\-(u_{2})_{\xi\xi}+c(u_{2})_{\xi} & = & ({\displaystyle \frac{1}{1+\epsilon_{2}}-u_{2})(bu_{1}-\epsilon_{1}(1+\epsilon_{2})u_{2})}\end{array}\right.\label{eq:3.6}\end{equation}
with boundary conditions 

\begin{equation}
\left\{ \begin{array}{lll}
\lim_{\xi\rightarrow-\infty}(w(\xi),z(\xi)) & = & (0,{\displaystyle 0),}\\
\\\lim_{\xi\rightarrow\infty}(w(\xi),z(\xi)) & = & (1,{\displaystyle \frac{1}{1+\epsilon_{2}}}).\end{array}\right.\label{eq:3.7}\end{equation}

\medskip{}

\noindent
\textbf{Remark 3.1.$\quad$} Note that from hypotheses [H1]-[H3] and relations (3.2), we have the following inequalities: 

\medskip{}

           $0<\epsilon_{1}<b$, $0<r<1$, $\epsilon_{2}>0$
and $1-\frac{r}{1+\epsilon_{2}}>b-\epsilon_{1}>0$.

\medskip{}

Before showing the existence of the traveling wave solutions for (\ref{eq:3.6})
with boundary conditions (\ref{eq:3.7}), we first recall the following
well known fact: (please see \cite{KolmogorovPetrovskiiPiskunov},
\cite{Sattinger} for the proof)

\medskip{}

Let a function $f$ be a $C^{2}$ function on the interval $[0,\beta], \beta>0$,
with $f>0$ on $(0,\beta)$, and $f(0)=f(\beta)=0$, $f'(0)=\alpha_{1}>0$,
$f'(\beta)=-\beta_{1}<0$.

\begin{lem}
\label{lem:3.1}Corresponding to every $c\geq2\sqrt{\alpha_{1}}$, the
boundary value problem\begin{equation}
\left\{ \begin{array}{l}
\omega''(\xi)-c\omega'(\xi)+f(\omega(\xi))=0,\\
\\\omega(-\infty)=0,\quad\omega(+\infty)=b.\end{array}\quad\quad\xi\in\mathbb{R}\right.\label{eq:3.8}\end{equation}
has a unique monotonically increasing traveling wave solution $\omega_{c}(\xi)$
, $\xi\in\mathbb{R}$ , where the lower index denotes the dependence
of the wave solution $\omega$ on $c$.
\end{lem}
We next show the existence of the traveling wave solution for system
\eqref{eq:3.6}-\eqref{eq:3.7}.

\vspace{7pt}

\begin{thm}
\label{thm: 3.2}Let the parameters $\epsilon_{1}$, $\epsilon_{2}$,
$b$ and $r$ satisfy conditions in Remark 3.1, then corresponding to every $c\geq2\sqrt{1-\frac{r}{1+\epsilon_{2}}}$
system (\ref{eq:3.6}) has a monotone traveling wave solution satisfying
the boundary condition (\ref{eq:3.7}). (Recall hypotheses [H1]-[H3] imply all the conditions in Remark 3.1 are valid.)
\end{thm}
\begin{proof}
The proof will be done by monotone iterating a pair of smooth upper-
and lower-solutions. We first construct a twice differentiable smooth
upper-solutions. 

According to lemma \ref{lem:3.1}, for every $c\geq2\sqrt{1-\frac{r}{1+\epsilon_{2}}}$,
there is correspondingly a montonically increasing $C^{2}$ function
$Y(\xi)$, $\xi\in\mathbb{R}$ satisfying

\noindent \begin{equation}
\left\{ \begin{array}{l}
Y_{\xi\xi}-cY_{\xi}+(1-\frac{r}{1+\epsilon_{2}})Y(1-Y)=0,\\
\\Y(-\infty)=0,\quad Y(\infty)=1.\end{array}\right.\label{eq:3.9}\end{equation}

Define

\noindent \begin{eqnarray}
\bar{u}_{1}(\xi)=Y(\xi) & , & \bar{u}_{2}(\xi)=\frac{1}{1+\epsilon_{2}}Y(\xi).\label{eq: 3.10}\end{eqnarray}
 For $0\leq u_{2}(\xi)\leq\bar{u}_{2}(\xi)$, we readily verify that
\begin{equation}
\begin{array}{ll}
 & -\bar{u}_{1}''(\xi)+c\bar{u}_{1}'(\xi)-\bar{u}_{1}({\displaystyle \frac{1+\epsilon_{2}-r}{1+\epsilon_{2}}-\bar{u}_{1}+ru_{2})}\\
\\= & (1-{\displaystyle \frac{r}{1+\epsilon_{2}})\bar{u}_{1}(1-\bar{u}_{1})-\bar{u}_{1}(1-{\displaystyle \frac{r}{1+\epsilon_{2}}-\bar{u}_{1}+ru_{2})}}\\
\\= & \bar{u}_{1}\left\{ (1-{\displaystyle \frac{r}{1+\epsilon_{2}})(1-\bar{u}_{1})-1+{\displaystyle \frac{r}{1+\epsilon_{2}}+\bar{u}_{1}-ru_{2}}}\right\} \\
\\= & \bar{u}_{1}\left\{ {\displaystyle \frac{r}{1+\epsilon_{2}}\bar{u}_{1}-ru_{2}}\right\} \\
\\\geq & \bar{u}_{1}\left\{ {\displaystyle \frac{r}{1+\epsilon_{2}}\bar{u}_{1}-r\bar{u}_{2}}\right\} \equiv0\end{array}\label{eq: 3.11}\end{equation}
 for all $-\infty<\xi<\infty$. For $0\leq u_{1}(\xi)\leq\bar{u}_{1}(\xi)$,
one verifies \begin{equation}
\begin{array}{ll}
 & -\bar{u}_{2}''(\xi)+c\bar{u}_{2}'(\xi)-({\displaystyle \frac{1}{1+\epsilon_{2}}-\bar{u}_{2})(bu_{1}-\epsilon_{1}(1+\epsilon_{2})\bar{u}_{2})}\\
\\= & {\displaystyle \frac{1}{1+\epsilon_{2}}\left\{ -Y''+cY'-(1-Y)(bu_{1}-\epsilon_{1}Y)\right\} }\\
\\= & {\displaystyle \frac{1}{1+\epsilon_{2}}\left\{ (1-{\displaystyle \frac{r}{1+\epsilon_{2}})Y(1-Y)+(1-Y)(\epsilon_{1}Y-bu_{1})}\right\} }\\
\\\geq & {\displaystyle \frac{1}{1+\epsilon_{2}}(1-Y)\left\{ (1-{\displaystyle \frac{r}{1+\epsilon_{2}})Y+\epsilon_{1}Y-bY}\right\} }\\
= & {\displaystyle \frac{1}{1+\epsilon_{2}}(1-Y)Y\left\{ 1-{\displaystyle \frac{r}{1+\epsilon_{2}}+\epsilon_{1}-b}\right\} \ge0.}\\
\\\end{array}\label{eq: 3.12}\end{equation}
 The last inequality is true provided $b<1-\frac{r}{1+\epsilon_{2}}+\epsilon_{1}$,
which is valid due to hypothesis H. It is also straightforward to
verify that $(\bar{u}_{1},\bar{u}_{2})$ satisfies the boundary conditions
(\ref{eq:3.7}).

\noindent We next construct a twice continuously differentiable lower
solution for the system (\ref{eq:3.6})-(\ref{eq:3.7}). Let the function
$Z(\xi)$, $\xi\in\mathbb{R}$ be the solution of 

\noindent \begin{equation}
\left\{ \begin{array}{l}
Z_{\xi\xi}-cZ_{\xi}+(1-\frac{r}{1+\epsilon_{2}})Z(1-\frac{1-\frac{lr}{1+\epsilon_{2}}}{1-\frac{r}{1+\epsilon_{2}}}Z)=0,\\
\\Z(-\infty)=0,\quad Z(\infty)=\frac{1-\frac{r}{1+\epsilon_{2}}}{1-\frac{lr}{1+\epsilon_{2}}}.\end{array}\right.\label{eq:3.13}\end{equation}
 Here $l$ is some number in the interval $(0,1)$ to be determined.
One can readily verify that the solutions of (\ref{eq:3.9}) and (\ref{eq:3.13})
are related by the following \begin{equation}
Z(\xi)=\frac{1-\frac{r}{1+\epsilon_{2}}}{1-\frac{lr}{1+\epsilon_{2}}}Y(\xi),\,\,\,\,\xi\in\mathbb{R}.\label{eq:3.14}\end{equation}
 Since $0<l<1$, we have \begin{equation}
Z(\xi)<Y(\xi),\,\,\,\xi\in\mathbb{R}.\label{eq:3.15}\end{equation}

\noindent We define a lower solution of (\ref{eq:3.6}), (\ref{eq:3.7})
by setting \begin{equation}
\tilde{u}_{1}=Z,\,\,\,\tilde{u}_{2}=\frac{l}{1+\epsilon_{2}}Z,\label{eq:3.16}\end{equation}
 where $l\in(0,1)$ is to be determined. We readily verify that they
satisfy \begin{equation}
\begin{array}{ll}
 & -\tilde{u}_{1}''(\xi)+c\tilde{u}_{1}'(\xi)-\tilde{u}_{1}({\displaystyle \frac{1+\epsilon_{2}-r}{1+\epsilon_{2}}-\tilde{u}_{1}+r\tilde{u}_{2})}\\
\\= & Z\left\{ (1-{\displaystyle \frac{r}{1+\epsilon_{2}})-(1-{\displaystyle \frac{lr}{1+\epsilon_{2}})Z-(\frac{1+\epsilon_{2}-r}{1+\epsilon_{2}})+Z-\frac{rl}{1+\epsilon_{2}}Z}}\right\} \\
\\= & 0.\end{array}\label{eq:3.17}\end{equation}
 Moreover, we have \vspace{3pt}
 \begin{equation}
\begin{array}{ll}
 & -\tilde{u}_{2}''(\xi)+c\tilde{u}_{2}'(\xi)-({\displaystyle \frac{1}{1+\epsilon_{2}}-\tilde{u}_{2})(b\tilde{u}_{1}-\epsilon_{1}(1+\epsilon_{2})\tilde{u}_{2})}\\
\\= & {\displaystyle \frac{l}{1+\epsilon_{2}}Z\left\{ 1-\frac{r}{1+\epsilon_{2}}-(1-\frac{lr}{1+\epsilon_{2}})Z\right\} -(\frac{1}{1+\epsilon_{2}}-\frac{l}{1+\epsilon_{2}}Z)\left\{ bZ-\epsilon_{1}lZ\right\} }\\
\\= & {\displaystyle \frac{l}{1+\epsilon_{2}}Z\left\{ (1-Z)-r(\frac{1}{1+\epsilon_{2}}-\frac{l}{1+\epsilon_{2}}Z)\right\} -(\frac{1}{1+\epsilon_{2}}-\frac{l}{1+\epsilon_{2}}Z)\left\{ bZ-\epsilon_{1}lZ\right\} }\\
\\\le & {\displaystyle (\frac{1}{1+\epsilon_{2}}-\frac{l}{1+\epsilon_{2}}Z)\left\{ -\frac{rl}{1+\epsilon_{2}}Z-bZ+\epsilon_{1}lZ\right\} +(1-lZ)\frac{lZ}{1+\epsilon_{2}}}\\
\\= & {\displaystyle (\frac{1}{1+\epsilon_{2}}-\frac{l}{1+\epsilon_{2}}Z)\left\{ -\frac{rl}{1+\epsilon_{2}}-b+\epsilon_{1}l+l\right\} Z}\\
\\\le & 0.\end{array}\label{eq:3.18}\end{equation}
 The last inequality is valid provided that \begin{equation}
0<l\le\min\{1,\frac{b}{1+\epsilon_{1}-\frac{r}{1+\epsilon_{2}}}\}.\label{eq:3.19}\end{equation}
Also by the limiting boundary conditions of (\ref{eq:3.13}) we see
$(\tilde{u}_{1},\tilde{u}_{2})(-\infty)=(0,0)$ and $(\tilde{u}_{1},\tilde{u}_{2})(+\infty)=(\frac{1+\epsilon_{2}-r}{1+\epsilon_{2}-lr},\frac{l(1+\epsilon_{2}-r)}{(1+\epsilon_{2})(1+\epsilon_{2}-lr)}).$
Inequality (\ref{eq:3.17}) along with (\ref{eq:3.18}) show that
$(\tilde{u}_{1},\tilde{u}_{2})$ consists of a pair of lower-solutions
for system (\ref{eq:3.6}), (\ref{eq:3.7}). 

Noting that such constructed upper- and lower-solution pairs are ordered.
We can apply the monotone iteration methods provided in \cite{WuZou}
or \cite{BoumenirNguyen} to derive the conclusion of this Theorem.
Here we only sketch the ideas. 

Following the notions in \cite{WuZou}, we write $\beta=diag.(0,0)$,
${\bf K}=(K_{1},K_{2}):=(1,\frac{1}{1+\epsilon_{2}})$, lower-solution
$\tilde{\rho}(\xi)=(Z(\xi),\frac{l}{1+\epsilon_{2}}Z(\xi))$, and
upper-solution 
$\bar{\rho}(\xi)=(Y(\xi),$\,\,\,\,\,\,\,\,\,
$\frac{1}{1+\epsilon_{2}}Y(\xi))$,
$\xi\in\mathbb{R}$,
as described above. As explained below, a slight variant of Theorem
3.6 or Theorem 3.6' in \cite{WuZou} is needed, because $(F_{1},F_{2})$
defined in (\ref{eq:3.22}) below has an additional zero at $(0,\frac{1}{1+\epsilon_{2}})$
between ${\bf 0}:=(0,0)$ and ${\bf K}$. By means of the iterative
procedure in the proof of Theorem 3.6 in \cite{WuZou}, we first obtain
a solution of \eqref{eq:3.6} $\phi(\xi):=(u_{1}(\xi),u_{2}(\xi))$,
satisfying the inequality \begin{equation}
\tilde{\rho}(\xi)\le\phi(\xi)\le\bar{\rho}(\xi),\quad\xi\in\mathbb{R}.\label{eq:3.20}\end{equation}
 From the comparison argument with $\bar{\rho}(\xi)$ in the proof
of Theorem 3.6 in \cite{WuZou}, we have \[
\lim_{\xi\rightarrow-\infty}u_{1}(\xi)=\lim_{\xi\rightarrow-\infty}u_{2}(\xi)=0.\]
 Again, by the limit and comparison argument in the proof of Theorem
3.6 in \cite{WuZou}, we obtain for $i=1,2$, \begin{equation}
\lim_{\xi\rightarrow\infty}u_{i}(\xi)=Q_{i},\qquad F_{i}(Q_{1},Q_{2})=0,\label{eq:3.21}\end{equation}
 where \begin{equation}
\begin{array}{l}
F_{1}(\rho_{1},\rho_{2})=\rho_{1}(\frac{1+\epsilon_{2}-r}{1+\epsilon_{2}}-\rho_{1}+r\rho_{2}),\\
\\F_{2}(\rho_{1},\rho_{2})=(\frac{1}{1+\epsilon_{2}}-\rho_{2})(b\rho_{1}-\epsilon_{1}(1+\epsilon_{2})\rho_{2});\end{array}\label{eq:3.22}\end{equation}
 and \begin{equation}
0<\lim_{\xi\rightarrow\infty}Z(\xi)\le Q_{1}\le K_{1}=1,\quad0<\lim_{\xi\rightarrow\infty}\frac{l}{1+\epsilon_{2}}Z(\xi)\le Q_{2}\le K_{2}=\frac{1}{1+\epsilon_{2}}.\label{eq:3.23}\end{equation}
 We then deduce from condition (H) and \eqref{eq:3.20}, \eqref{eq:3.21},
\eqref{eq:3.23} that we must have \begin{equation}
Q_{1}=K_{1}=1,\quad Q_{2}=K_{2}=\frac{1}{1+\epsilon_{2}}.\label{eq:3.24}\end{equation}

\end{proof}
\vspace{7pt}

\noindent \textbf{Remark 3.2.}  One can translate Theorem 3.2 into the following: Assuming
hypotheses {[}H1], {[}H2], {[}H3], then for any $c\ge2\sqrt{1-\frac{c_{1}a_{2}}{c_{2}a_{1}}}$.
system \eqref{eq:1.1} has a traveling wave solution connecting $(0,\frac{a_{2}}{c_{2}})$
with $(\frac{a_{1}}{b_{1}},0)$ as the variable $\sqrt{\frac{a_{1}}{d}}x+ca_{1}t$
running from $-\infty$ to $+\infty$.

\noindent \textbf{Remark 3.3.} Theorem 3.2 does not insure the \textit{strict} monotonicity
of the resulting traveling wave solutions, as the iteration is only
applied for the upper-solution, and the lower solution is served as
a nonzero barrier so that the iteration limit does not coverge to
zero.

\vspace{7pt}

To further study the asymptotics of the traveling wave solutions as
obtained in Theorem \ref{thm: 3.2}, we shall need the following Lemma concerning the scalar problem (3.8).

\begin{lem}
\label{lem: 3.3}The solution $w_{c}(\xi)$ to (3.8), described in Lemma 3.1,
has the following asymptotic behaviors:

1. Corresponding to the wave speed $c > 2\sqrt{\alpha_{1}}$,

\begin{equation}
\left\{ \begin{array}{l}
\omega_{c}(\xi)=a_{\omega}e^{\frac{c-\sqrt{c^{2}-4\alpha_{1}}}{2}\xi}+o(e^{\frac{c-\sqrt{c^{2}-4\alpha_{1}}}{2}\xi}),\quad\mbox{ as }\xi\rightarrow-\infty\\
\\\omega_{c}(\xi)=\beta-b_{\omega}e^{\frac{c-\sqrt{c^{2}+4\beta_{1}}}{2}\xi}+o(e^{\frac{c-\sqrt{c^{2}+4\beta_{1}}}{2}\xi}),\quad\mbox{ as }\xi\rightarrow+\infty,\end{array}\right.\label{eq:3.25}\end{equation}
where $a_{\omega}$ and $b_{\omega}$ are positive constants;

2. Corresponding to minimal wave speed $c=2\sqrt{\alpha_{1}}$,

\begin{equation}
\left\{ \begin{array}{l}
\omega_{c}(\xi)=(a_{c}+d_{c}\xi)e^{\sqrt{\alpha_{1}}\xi}+o(\xi e^{\sqrt{\alpha_{1}}\xi}),\quad\mbox{ as }\xi\rightarrow-\infty\\
\\\omega_{c}(\xi)=\beta-b_{c}e^{\frac{c-\sqrt{c^{2}+4\beta_{1}}}{2}\xi}+o(e^{\frac{c-\sqrt{c^{2}+4\beta_{1}}}{2}\xi}),\quad\mbox{ as }\xi\rightarrow+\infty;\end{array}\right.\label{eq:3.26}\end{equation}
where the constant $d_{c}<0$, $b_{c}>0$ and $a_{c}\in\mathbb{R}$.
\end{lem}
\begin{proof}
The conclusion follows \cite{Sattinger}, \cite{Thiery} with slight
changes.
\end{proof}
\noindent Based on Lemma \ref{lem: 3.3}, we study the asymptotic
behaviors of the traveling wave solutions of system (\ref{eq:3.6}),
(\ref{eq:3.7}) at infinities.

\begin{cor}
\label{cor:3.4}Assume hypotheses [H1] to [H3], and thus all the conditions in Remark 3.1 are satisfied. Let $\alpha=1-\frac{r}{1+\epsilon_{2}}$, then the traveling wave solutions $(u_{1}(\xi),u_{2}(\xi))$
of system (\ref{eq:3.6}), (\ref{eq:3.7}) as obtained in Theorem
\ref{thm: 3.2} have the following asymptotic behaviors: 

1. Corresponding to each wave speed $c>2\sqrt{\alpha}$, the traveling
wave solution $(u_{1}(\xi),u_{2}(\xi))$ satisfies \begin{equation}
\left(\begin{array}{c}
u_{1}(\xi)\\
\\u_{2}(\xi)\end{array}\right)=\left(\begin{array}{c}
A_{1}\\
\\A_{2}\end{array}\right)e^{\frac{c-\sqrt{c^{2}-4\alpha}}{2}\xi}+o(e^{\frac{c-\sqrt{c^{2}-4\alpha}}{2}\xi})\label{eq:3.27}\end{equation}
 as $\xi\rightarrow-\infty$; and \begin{equation}
\left(\begin{array}{c}
u_{1}(\xi)\\
\\u_{2}(\xi)\end{array}\right)=\left(\begin{array}{c}
1\\
\\{\displaystyle \frac{1}{1+\epsilon_{2}}}\end{array}\right)-\left(\begin{array}{c}
\bar{A}_{1}\\
\\\bar{A}_{2}\end{array}\right)e^{\frac{c-\sqrt{c^{2}+4(b-\epsilon_{1})}}{2}\xi}+o(e^{\frac{c-\sqrt{c^{2}+4(b-\epsilon_{1})}}{2}\xi})\label{eq:3.28}\end{equation}
as $\xi\rightarrow+\infty,$ where $A_{1}$, $A_{2}$, $\bar{A}_{1}$,
$\bar{A}_{2}$ are positive constants; 

2. Corresponding to the wave speed $c_{\mbox{critical}}=2\sqrt{\alpha}$,
the traveling wave solution $(u_{1}(\xi),u_{2}(\xi))$ satisfies \begin{equation}
\left(\begin{array}{c}
u_{1}(\xi)\\
\\u_{2}(\xi)\end{array}\right)=\left(\begin{array}{c}
A_{11c}+A_{12c}\xi\\
\\A_{21c}+A_{22c}\xi\end{array}\right)e^{\sqrt{\alpha}\xi}+o(\xi e^{\sqrt{\alpha}\xi})\label{eq:3.29}\end{equation}
as $\xi\rightarrow-\infty$; and \begin{equation}
\left(\begin{array}{c}
u_{1}(\xi)\\
\\u_{2}(\xi)\end{array}\right)=\left(\begin{array}{c}
1\\
\\{\displaystyle \frac{1}{1+\epsilon_{2}}}\end{array}\right)-\left(\begin{array}{c}
\bar{A}_{11}\\
\\\bar{A}_{22}\end{array}\right)e^{\frac{c-\sqrt{c^{2}+4(b-\epsilon_{1})}}{2}\xi}+o(e^{\frac{c-\sqrt{c^{2}+4(b-\epsilon_{1})}}{2}\xi})\label{eq:3.30}\end{equation}
 as $\xi\rightarrow+\infty$, where $A_{12c},\, A_{22c}<0$, $A_{11c}$,
$A_{21c}\in\mathbb{R}$ and $\bar{A}_{11},$$\bar{A}_{22}>0$. 
\end{cor}
\noindent \vspace{7pt}

\begin{proof}
For $c=2\sqrt{1-\frac{r}{1+\epsilon_{2}}}$, according to Lemma \ref{lem: 3.3}
the upper-solution $(\bar{u}_{1},\bar{u}_{2})$ and the lower-solution
$(\tilde{u}_{1},\tilde{u}_{2})$ as defined in (\ref{eq: 3.10}),
(\ref{eq:3.16}) have the following respective asymptotic behaviors
at $-\infty$, 

\[
\left(\begin{array}{c}
\bar{u}_{1}\xi)\\
\\\bar{u}_{2}(\xi)\end{array}\right)=\left(\begin{array}{c}
\bar{A}_{11c}+\bar{A}_{12c}\xi\\
\\\bar{A}_{21c}+\bar{A}_{22c}\xi\end{array}\right)e^{\sqrt{\alpha}\xi}+o(\xi e^{\sqrt{\alpha}\xi}),\]
and 

\[
\left(\begin{array}{c}
\tilde{u}_{1}\xi)\\
\\\tilde{u}_{2}(\xi)\end{array}\right)=\left(\begin{array}{c}
\tilde{B}_{11c}+\tilde{B}_{12c}\xi\\
\\\tilde{B}_{21c}+\tilde{B}_{22c}\xi\end{array}\right)e^{\sqrt{\alpha}\xi}+o(\xi e^{\sqrt{\alpha}\xi});\]

While with wave speed $c>2\sqrt{1-\frac{r}{1+\epsilon_{2}}}$, 

\[
\left(\begin{array}{c}
\bar{u}_{1}\xi)\\
\\\bar{u}_{2}(\xi)\end{array}\right)=\left(\begin{array}{c}
\bar{A}_{1}\\
\\\bar{A}_{2}\end{array}\right)e^{\frac{c-\sqrt{c^{2}-4\alpha}}{2}\xi}+o(e^{\frac{c-\sqrt{c^{2}-4\alpha}}{2}\xi}),\]
and

\[
\left(\begin{array}{c}
\tilde{u}_{1}\xi)\\
\\\tilde{u}_{2}(\xi)\end{array}\right)=\left(\begin{array}{c}
\tilde{A}_{1}\\
\\\tilde{A}_{2}\end{array}\right)e^{\frac{c-\sqrt{c^{2}-4\alpha}}{2}\xi}+o(e^{\frac{c-\sqrt{c^{2}-4\alpha}}{2}\xi}).\]
where $\bar{A}_{i}$, $\tilde{A}_{i}$, $-\bar{A}_{i2c}$, $-\tilde{B}_{i2c}$
$i=1,2$ are positive constants and $\bar{A}_{i1c}$, $\tilde{B}_{i1c}$
are constants.

Noting the upper- and lower-solutions have the same asymptotic growth
rate at $-\infty$, we immediately have (\ref{eq:3.27}) and (\ref{eq:3.29})
via comparison. 

We next derive the asymptotic behaviors of the traveling wave solutions
of (\ref{eq:3.6}), (\ref{eq:3.7}) at $+\infty$. Since the the traveling
wave solution $(u_{1}(\xi),u_{2}(\xi))^{T}$ monotonically aproaches
the steady state $(1,\frac{1}{1+\epsilon_{2}})$ as $\xi\rightarrow+\infty$,
by letting $(w_{1}(\xi),w_{2}(\xi))$ be the derivative of the traveling
wave solution, we have $(w_{1},w_{2})(+\infty)=(0,0)$ \cite{WuZou}.

We linearize system (\ref{eq:3.6}) about the traveling wave solution
$(u_{1}(\xi),u_{2}(\xi))$ to obtain 

\begin{equation}
\left\{ \begin{array}{l}
\phi_{\xi\xi}-c\phi_{\xi}+(1-\frac{r}{1+\epsilon_{2}}-2u_{1}+ru_{2})\phi+ru_{1}\psi=0,\\
\\\psi_{\xi\xi}-c\psi_{\xi}+b(\frac{1}{1+\epsilon_{2}}-u_{2})\phi+[-\epsilon_{1}-bu_{1}+2\epsilon_{1}(1+\epsilon_{2})u_{2}]\psi=0.\end{array}\right.\label{eq:3.31}\end{equation}
It is easy to see that $(w_{1},w_{2})$ solves \eqref{eq:3.31}. 

The limit system of \eqref{eq:3.31} at $+\infty$ is a constant coefficient
system, and is given by 

\begin{equation}
\left\{ \begin{array}{l}
(\phi^{+})_{\xi\xi}-c(\phi^{+})_{\xi}-\phi^{+}+r\psi^{+}=0,\\
\\(\psi^{+})_{\xi\xi}-c(\psi^{+})_{\xi}+(\epsilon_{1}-b)\psi^{+}=0.\end{array}\right.\label{eq:3.32}\end{equation}

The exponential growth rates of the traveling wave solutions of \eqref{eq:3.6}-\eqref{eq:3.7}
are determined by those of the solutions of \eqref{eq:3.32}. The
justification is as follows: first we note that \eqref{eq:3.32} admits
exponential dichotomy. By the roughness of exponential dichotomy,
solutions of \eqref{eq:3.31} grow/decay exponentially (possibly with
a different exponential rate) \cite{Coddington}, \cite{Sandstede}.
Since the derivative $(w_{1},w_{2})$ of the traveling wave solution
solves \eqref{eq:3.31}, the traveling wave solutions of (\ref{eq:3.6})
approach exponentially to the steady state $(1,\frac{1}{1+\epsilon_{2}})$. 

To find out the exact asymptotic rates of the solutions of \eqref{eq:3.31},
we first change \eqref{eq:3.31}, \eqref{eq:3.32} into the first
order systems. letting $\phi_{\xi}=\phi_{1}$, $\psi_{\xi}=\psi_{1}$,
$\phi_{\xi}^{+}=\phi_{1}^{+}$, $\psi_{\xi}^{+}=\psi_{1}^{+}$, we
have

\[
\frac{d}{d\xi}\left(\begin{array}{c}
\phi\\
\phi_{1}\\
\psi\\
\psi_{1}\end{array}\right)=\mathcal{R}^{+}\left(\begin{array}{c}
\phi\\
\phi_{1}\\
\psi\\
\psi_{1}\end{array}\right)\]
and

\[
\frac{d}{d\xi}\left(\begin{array}{c}
\phi^{+}\\
\phi_{1}^{+}\\
\psi^{+}\\
\psi_{1}^{+}\end{array}\right)=\mathcal{R}^{\infty}\left(\begin{array}{c}
\phi^{+}\\
\phi_{1}^{+}\\
\psi^{+}\\
\psi_{1}^{+}\end{array}\right).\]
where \[
\mathcal{R}^{+}=\left(\begin{array}{cccc}
0 & 1 & 0 & 0\\
-1+\frac{r}{1+\epsilon_{2}}+2u_{1}(\xi)-ru_{2}(\xi) & c & -ru_{1}(\xi) & 0\\
0 & 0 & 0 & 1\\
-b(\frac{1}{1+\epsilon_{2}}-u_{2}(\xi)) & 0 & \epsilon_{1}+bu_{1}(\xi)-2\epsilon_{1}(1+\epsilon_{2})u_{2}(\xi) & c\end{array}\right)\]
and

\[
\mathcal{R}^{\infty}=\left(\begin{array}{cccc}
0 & 1 & 0 & 0\\
1 & c & -r & 0\\
0 & 0 & 0 & 1\\
0 & 0 & b-\epsilon_{1} & c\end{array}\right).\]

The exponential growth of the traveling wave solution $(u_{1}(\xi),u_{2}(\xi))$
at infinity implies that 

\[
\int_{t_{0}}^{\infty}|\mathcal{R}^{+}-\mathcal{R}^{\infty}|d\xi<+\infty\quad\mbox{for}\;\mbox{some large}\: t_{0}>0.\]
It is easy to check that the matrix $\mathcal{R}^{\infty}$ has 4
distinct eigenvalues, then by \cite{Coddington}and \cite{Sattinger},
the asymptotic exponential rates of the solutions of \eqref{eq:3.31}
are the same as those of the solutions of \eqref{eq:3.32}.

We search for the solutions of \eqref{eq:3.32} with zero limit at
$+\infty$. Owing to conditions in Remark 3.1, the second equation of \eqref{eq:3.32}
has general solution of the form:

\begin{equation}
\psi^{+}(\xi)=c_{21}e^{\frac{c-\sqrt{c^{2}+4(b-\epsilon_{1})}}{2}\xi}+c_{22}e^{\frac{c+\sqrt{c^{2}+4(b-\epsilon_{1})}}{2}\xi}\label{eq:3.33}\end{equation}
with $\frac{c-\sqrt{c^{2}+4(b-\epsilon_{1})}}{2}<0$, $\frac{c+\sqrt{c^{2}+4(b-\epsilon_{1})}}{2}>0$
and $c_{21},$ $c_{22}$ two constants. 

Substituting \eqref{eq:3.33} into the first equation of \eqref{eq:3.32},
we then have

\begin{equation}
(\phi^{+})_{\xi\xi}-c(\phi^{+})_{\xi}-\phi^{+}=-r(c_{21}e^{\frac{c-\sqrt{c^{2}+4(b-\epsilon_{1})}}{2}\xi}+c_{22}e^{\frac{c+\sqrt{c^{2}+4(b-\epsilon_{1})}}{2}\xi}),\label{eq:3.34}\end{equation}
with general solution of the form

\begin{equation}
\phi^{+}(\xi)=c_{11}e^{\frac{c+\sqrt{c^{2}+4}}{2}\xi}+c_{12}e^{\frac{c-\sqrt{c^{2}+4}}{2}\xi}+\bar{c}_{21}e^{\frac{c-\sqrt{c^{2}+4(b-\epsilon_{1})}}{2}\xi}+\bar{c}_{22}e^{\frac{c+\sqrt{c^{2}+4(b-\epsilon_{1})}}{2}\xi}.\label{eq:3.35}\end{equation}

Comparing (3.31), (3.32) and using the fact that $(w_{1},w_{2})$ is a solution of (3.31) with $(w_{1}(+\infty),w_{2}(+\infty))=(0,0)$, we deduce from (3.35) that as $\xi \rightarrow  +\infty$,

$$
\begin{array}{l}
w_{1}(\xi)= (\hat{c}_{12}+o(1))e^{\frac{c-\sqrt{c^{2}+4}}{2}\xi}+(\tilde{c}_{21}+o(1))e^{\frac{c-\sqrt{c^{2}+4(b-\epsilon_{1})}}{2}\xi}\\ \\
w_{2}(\xi)=(\hat{c}_{21}+o(1))e^{\frac{c-\sqrt{c^{2}+4(b-\epsilon_{1})}}{2}\xi}
\end{array}
$$
with $\tilde{c}_{21}\neq 0$ and $\hat{c}_{21}\neq 0$.  Since $\frac{c-\sqrt{c^{2}+4}}{2}<\frac{c-\sqrt{c^{2}+4(b-\epsilon_{1})}}{2}<0$, by [H3], we may write
$$
w_{1}(\xi)= (\tilde{c}_{21}+o(1))e^{\frac{c-\sqrt{c^{2}+4(b-\epsilon_{1})}}{2}\xi}
$$
as $\xi \rightarrow +\infty$.

We therefore
conclude this proof by integrating $(w_{1}(\xi),w_{2}(\xi))^{T}$
from $\xi$ to $+\infty$ with $\xi$ sufficiently large. 
\end{proof}
We next derive a corollary that is important to section \ref{sec:4}. 

\begin{cor}
\label{cor:3.5}Assume the hypotheses of Corollary 3.4. For every wave speed with $c\geq2\sqrt{1-\frac{r}{1+\epsilon_{2}}}$,
the corresponding traveling wave solution $(u_{1}(\xi),u_{2}(\xi))^{T}$ of
system (\ref{eq:3.6}), (\ref{eq:3.7}) obtained in Corollary 3.4 is a strict monotonic function
for $\xi\in\mathbb{R}$. 
\end{cor}
\begin{proof}
According to the monotone iteration (\cite{WuZou}), the traveling
wave solution $(u_{1}(\xi),u_{2}(\xi))^{T}$ to (\ref{eq:3.6}), (\ref{eq:3.7})
is monotone, which implies that $(u_{1}'(\xi),u_{2}'(\xi))^{T}\geq0$
for $\xi\in\mathbb{R}$. By Corollary \ref{cor:3.4}, one sees that
$(u_{1}'(\pm\infty),u_{2}'(\pm\infty))^{T}=0$. The monotonicity of
system (\ref{eq:3.6}) and the maximum principle lead to the conclusion
that $(u_{1}'(\xi),u_{2}'(\xi))^{T}>0$ for $\xi\in\mathbb{R}$.
\end{proof}
We next show the uniqueness of the traveling wave solutions.

\begin{thm}
\label{thm:3.6}Assume hypotheses [H1] to [H3]. The traveling wave solution to system (\ref{eq:3.6})-(\ref{eq:3.7}),
obtained for each wave speed $c\geq2\sqrt{1-\frac{r}{1+\epsilon_{2}}}$, with properties described in Corollary 3.4 and 3.5,
is unique up to a translation of the origin.
\end{thm}
\begin{proof}
We only prove the conclusion for traveling wave solutions with asymptotic
behaviors \eqref{eq:3.27} and \eqref{eq:3.28}, since other case
can be proved similarly. Let $U_{1}(\xi)$ and $U_{2}(\xi)$ be traveling
wave solutions of system \eqref{eq:3.6}-\eqref{eq:3.7}, with speed $c>2\sqrt{\alpha}$ and properpties described.
There exist positive constants $A_{i}$, $B_{i}$, $i=1,2,3,4$ and
a large number $N>0$ such that for $\xi<-N$,\begin{equation}
U_{1}(\xi)=\left(\begin{array}{c}
(A_{1}+o(1))e^{\frac{c-\sqrt{c^{2}-4\alpha}}{2}\xi}\\
\\(A_{2}+o(1))e^{\frac{c-\sqrt{c^{2}-4\alpha}}{2}\xi}\end{array}\right)\label{eq:3.36}\end{equation}
 \begin{equation}
U_{2}(\xi)=\left(\begin{array}{c}
(A_{3}+o(1))e^{\frac{c-\sqrt{c^{2}-4\alpha}}{2}\xi}\\
\\(A_{4}+o(1))e^{\frac{c-\sqrt{c^{2}-4\alpha}}{2}\xi}\end{array}\right);\label{eq:3.37}\end{equation}
 and for $\xi>N$,\begin{equation}
U_{1}(\xi)=\left(\begin{array}{c}
{\displaystyle 1-(B_{1}+o(1))e^{\frac{c-\sqrt{c^{2}+4(b-\epsilon_{1})}}{2}\xi}}\\
\\{\displaystyle \frac{1}{1+\epsilon_{2}}-(B_{2}+o(1))e^{\frac{c-\sqrt{c^{2}+4(b-\epsilon_{1})}}{2}\xi}}\end{array}\right)\label{eq:3.38}\end{equation}
 \begin{equation}
U_{2}(\xi)=\left(\begin{array}{c}
{\displaystyle 1-(B_{3}+o(1))e^{\frac{c-\sqrt{c^{2}+4(b-\epsilon_{1})}}{2}\xi}}\\
\\{\displaystyle \frac{1}{1+\epsilon_{2}}-(B_{4}+o(1))e^{\frac{c-\sqrt{c^{2}+4(b-\epsilon_{1})}}{2}\xi}}\end{array}\right)\label{eq:3.39}\end{equation}
 The traveling wave solutions of system \eqref{eq:3.6} are translation
invariant, thus for any $\theta>0$, $U_{1}^{\theta}(\xi):=U_{1}(\xi+\theta)$
is also a traveling wave solution of \eqref{eq:3.6}. By \eqref{eq:3.36}
and \eqref{eq:3.38}, the solution $U_{1}(\xi+\theta)$ has the following
asymptotic behaviors:\begin{equation}
U_{1}^{\theta}(\xi)=\left(\begin{array}{c}
(A_{1}+o(1))e^{\frac{c+\sqrt{c^{2}-4\alpha}}{2}\theta}e^{\frac{c+\sqrt{c^{2}-4\alpha}}{2}\xi}\\
\\(A_{2}+o(1))e^{\frac{c+\sqrt{c^{2}-4\alpha}}{2}\theta}e^{\frac{c+\sqrt{c^{2}-4\alpha}}{2}\xi}\end{array}\right)\label{eq:3.40}\end{equation}
 for $\xi\leq-N-\theta$;\begin{equation}
U_{1}^{\theta}(\xi)=\left(\begin{array}{c}
{\displaystyle 1-(B_{1}+o(1))e^{\frac{c-\sqrt{c^{2}+4(b-\epsilon_{1})}}{2}\theta}e^{\frac{c-\sqrt{c^{2}+4(b-\epsilon_{1})}}{2}\xi}}\\
\\{\displaystyle \frac{1}{1+\epsilon_{2}}-(B_{2}+o(1))e^{\frac{c-\sqrt{c^{2}+4(b-\epsilon_{1})}}{2}\theta}e^{\frac{c-\sqrt{c^{2}+4(b-\epsilon_{1})}}{2}\xi}}\end{array}\right)\label{eq:3.41}\end{equation}
 for $\xi\geq N$.

It is clear that for $\theta$ large enough, we have \begin{equation}
A_{1}e^{\frac{c-\sqrt{c^{2}-4\alpha}}{2}\theta}>A_{3},\label{eq:3.42}\end{equation}
 \begin{equation}
A_{2}e^{\frac{c-\sqrt{c^{2}-4\alpha}}{2}\theta}>A_{4},\label{eq:3.43}\end{equation}
 \begin{equation}
B_{1}e^{\frac{c-\sqrt{c^{2}+4(b-\epsilon_{1})}}{2}\theta}<B_{3},\label{eq:3.44}\end{equation}
 \begin{equation}
B_{2}e^{\frac{c-\sqrt{c^{2}+4(b-\epsilon_{1})}}{2}\theta}<B_{4}.\label{eq:3.45}\end{equation}
 Thus for some $N>0$, formulas \eqref{eq:3.42} - \eqref{eq:3.45} imply that for $\theta$
large enough, \begin{equation}
U_{1}^{\theta}(\xi)>U_{2}(\xi)\label{eq:3.46}\end{equation}
 for $\xi\in(-\infty,-N]$$\cup$$[N+\infty).$ Here, the inequality
$">"$ in \eqref{eq:3.46} is component-wise. We now consider system
\eqref{eq:3.6} on $[-N,+N]$. First, suppose $U_{1}^{\theta}(\xi)\geq U_{2}(\xi)$
on $[-N,+N]$, then the function $W(\xi):=U_{1}^{\theta}(\xi)-U_{2}(\xi)\geq0$
and satisfies for some $\zeta\in(0,1)$: \begin{equation}
\begin{array}{ll}
W''-cW'+\left[\begin{array}{cc}
\frac{\partial F_{1}}{\partial u_{1}}(U_{2}+\zeta_{1}(U_{1}^{\theta}-U_{2})), & \frac{\partial F_{1}}{\partial u_{2}}(U_{2}+\zeta_{1}(U_{1}^{\theta}-U_{2}))\\
\\\frac{\partial F_{2}}{\partial u_{1}}(U_{2}+\zeta_{2}(U_{1}^{\theta}-U_{2})), & \frac{\partial F_{2}}{\partial u_{2}}(U_{2}+\zeta_{2}(U_{1}^{\theta}-U_{2}))\end{array}\right]W=0,\\
\\\hbox{for}\,\,\,\xi\in(-N,N),\,\,\,\,\hbox{and}\,\,\, W(-N)>0,\,\,\, W(+N)>0.\end{array}\label{eq:3.47}\end{equation}
Since the above system is monotone, we can readily deduce by maximum
principle that $W>0$, on $[-N,N]$. Consequently, we have $U_{1}^{\theta}(\xi)>U_{2}(\xi)$
on $\mathbb{R}$ in this case.

Second, suppose there are some points in $(-N,N)$ such that $U_{1}^{\theta}(\xi)<U_{2}(\xi)$.
We then increase $\theta$, that is shift $U_{1}^{\theta}(\xi)$ further
left so that $U_{1}^{\theta}(-N)>U_{2}(-N)$, $U_{1}^{\theta}(N)>U_{2}(N)$.
By the monotonicity of $U_{1}^{\theta}$ and $U_{2}$, we can find
a $\bar{\theta}\in(0,2N)$ such that in the interval $(-N,N)$, we
have $U_{1}^{\theta}(\xi+\bar{\theta})>U_{2}(\xi)$. Shifting $U_{1}^{\theta}(\xi+\bar{\theta})$
back until one component of $U_{1}^{\theta}(\xi+\bar{\theta})$ first
touches one component of $U_{2}(\xi)$ at some point $\bar{\bar{\xi}}\in(-N,N)$.
Then by maximum principle for that component again, we find that component
of $U_{1}^{\theta}$ and $U_{2}$ are identically equal for all $\xi\in[-N,N]$
for a larger $\theta$ than the original one such that \eqref{eq:3.46}
holds. This is a contradiction. Therefore, we must have\[
U_{1}^{\theta}(\xi)>U_{2}(\xi)\]
 for all $\xi\in R$, where $\theta$ is the one chosen by means of
\eqref{eq:3.42}-\eqref{eq:3.45} as described above.

Now, decrease $\theta$ until one of the following situations happens.

1. There exists $\bar{\theta}\geq0$, such that $U_{1}^{\bar{\theta}}(\xi)\equiv U_{2}(\xi)$.
In this case we have finished our proof.

2. For $\bar{\theta}\geq0$, there exists $\xi_{1}\in R$, such that
one of the components of $U^{\bar{\theta}}$ and $U_{2}$ are equal
at the point $\xi_{1}$; and for all $\xi\in\mathbb{R}$, we have
$U_{1}^{\bar{\theta}}(\xi)\geq U_{2}(\xi)$. We then consider the
system \eqref{eq:3.47} on $(-N,N)$ and $\theta=\bar{\theta}$ in
the definition for $W$. To fix ideas, we suppose that the first component
of $U_{1}^{\bar{\theta}}$ and $U_{2}$ is equal at the point $\xi_{1}$.
The maximum principle for this component implies that the first component
of $U_{1}^{\bar{\theta}}(\xi)$ is identically equal to that of $U_{2}(\xi)$.
Also, we readily obtain that for large $+\xi$, the limiting equation
for \eqref{eq:3.47} is the same as \eqref{eq:3.32} . Since the first
component of $W$ is identically zero and the off diagonal limit coefficient $r$ in the first equation in (3.32) is not equal to zero, we conclude from the first
equation in \eqref{eq:3.47} that the second component of $W$ must
vanish for all large $\xi$. By the maximum principle for the second
equation, we conclude that the second component of $W$ is also identically
zero for all $\xi\in\mathbb{R}$. Similarly, we next consider the case that the
second component of $U_{1}^{\bar{\theta}}$ and $U_{2}$ are equal
at the point $\xi_{1}$. We first obtain the limiting equation of (3.47) for large $-\xi$. The off diagonal limit coefficient of the second equation will be $\frac{b}{1+\epsilon_{2}}\neq 0$. We first deduce that the second componet of $W$  is indentically zero, and then the first component must also be identically zero for $\xi\in\mathbb{R}$.

Consequently, in either situation, there exists a $\bar{\theta}\geq0$,
such that \[
U_{1}^{\bar{\theta}}(\xi)\equiv U_{2}(\xi).\]
 for all $\xi\in\mathbb{R}$.
\end{proof}
\begin{thm}
\label{thm:3.7}Assume hypotheses [H1] to [H3]. System \eqref{eq:3.6}-\eqref{eq:3.7} does not have
strict monotonic traveling wave solution tending to $(0,0)^{T}$ as $\xi \rightarrow -\infty$ for $c<2\sqrt{\alpha}$. Here, $\alpha=1-\frac{r}{1+\epsilon_{2}}$.
\end{thm}
\begin{proof}
Suppose there is a constant $c$ with $0<c<2\sqrt{\alpha}$ and a corresponding
solution $V(\xi)=(v_{1}(\xi),v_{2}(\xi))^{T}$ of (\ref{eq:3.6}) tending to $(0,0)^{T}$ as $\xi \rightarrow \-\infty$. Similar to the proof of Corollary \ref{cor:3.4}, we can deduce by integrating the asymptotic approximation of its derivative that the asymptotic behaviors of  $V(\xi)$ at $-\infty$ must be of the form:  

\[
\left(\begin{array}{c}
v_{1}(\xi)\\
\\v_{2}(\xi)\end{array}\right)=\left(\begin{array}{c}
A_{s}\\
\\B_{s}\end{array}\right)e^{\frac{c-\sqrt{c^{2}-4\alpha}}{2}\xi}+\left(\begin{array}{c}
\bar{A_{s}}\\
\\\bar{B_{s}}\end{array}\right)e^{\frac{c+\sqrt{c^{2}-4\alpha}}{2}\xi}+h.o.t,\]
where $(A_{s},B_{s})^{T}$and $(\bar{A_{s},}\bar{B_{s}})$ can not
be both zero, and h.o.t. is the short notation for the higher order
terms. The condition $0<c<2\sqrt{\alpha}$ implies that $V(\xi)$
is oscillatting. This says that such solution of (\ref{eq:3.6}) with
$c<2\sqrt{\alpha}$ is not monotone.\textit{ }
\end{proof}

\section{\textbf{\label{sec:4}STABILITY OF THE TRAVELING WAVES WITH NON-CRITICAL
SPEEDS}}

\setcounter{equation}{0}

In this section, we always hypotheses [H1] to [H3] for system (1.1); thus all the conditions in Remark 3.1 are satisfied for (3.6) and subsequent systems. We first show that the traveling wave solutions with
the non-critical speed obtained in Theorem \ref{thm: 3.2} is unstable
in the space of continuous function $C(\mathbb{R})\times C(\mathbb{R})$
(Please see definitions below). This motivates us to investigate the
stability in the {}``smaller'' exponentially weighted Banach spaces.
We will concentrate on the stability of the traveling waves with non-critical
wave speeds.

The following set up of the problem is standard: Let $U=(u,v)^{T}$,
$F(U)=(u_{1}(\frac{1+\epsilon_{2}-r}{1+\epsilon_{2}}-u_{1}+ru_{2}),(\frac{1}{1+\epsilon_{2}}-u_{2})(bu_{1}-\epsilon_{1}(1+\epsilon_{2})u_{2}))^{T}$
and write system \eqref{eq:3.6} in the moving coordinates $\xi=x+ct$,
$c>0$. In terms of $(\xi,t)$ variable, (\ref{eq:3.6}) is changed
into 

\begin{equation}
\left\{ \begin{array}{ccl}
U_{t} & = & U_{\xi\xi}-cU_{\xi}+F(U),\\
U(\xi,0) & = & \bar{U},\end{array}\right.\label{eq: 4.1}\end{equation}
where $\bar{U}$ the initial function. Let $U^{*}(\xi)=(u^{*}(\xi),v^{*}(\xi))^{T}$
, $\xi=x+c^{*}t$ be the traveling wave solution of (\ref{eq:3.6})
with wave speed $c^{*}>2\sqrt{\alpha}$ . It is easy to see that $U^{*}(\xi)$
is a non-constant steady state of the system \eqref{eq: 4.1}. We
consider the perturbation of this steady state solution. By letting
$U(\xi,t)=U^{*}(\xi)+V(\xi,t)$ , we obtain the system for the perturbation
function $V$:

\begin{equation}
\left\{ \begin{array}{ccl}
V_{t} & = & LV+\mathcal{N}(V,U^{*}),\\
V(\xi,0) & = & \bar{U}(\xi)-U^{*}(\xi)\end{array}\right.\label{eq: 4.2}\end{equation}
 where \begin{equation}
LV=V_{\xi\xi}-c^{*}V_{\xi}+\frac{\partial F}{\partial U}(U^{*})V\label{eq: 4.3}\end{equation}
 is a linear operator, and \begin{equation}
\mathcal{N}(V,U^{*})=F(U^{*}+V)-F(U^{*})-\frac{\partial F}{\partial U}(U^{*})V\label{eq: 4.4}\end{equation}
 is a nonlinear operator. 

The stability of the traveling wave solution $U^{*}(\xi)$ in certain
Banach space is determined by the location of the spectrum, $\sigma(L)$, of $L$.

Let
$$ 
\sigma_{p}(L)= \,\{\lambda \in \sigma(L)\,|\,\lambda\,\, is\, an\,\, eigenvalue\,\, of\, L\},
$$
and $\sigma_{e}(L)$ be the essential spectrum of $L$, which are points in $\sigma(L)$ outside $\sigma_{p}(L)\cap \{isolated \,\, eigenvalues\,\,of\,\,$L$\,\, with \,\,finite\,\,multiplicity\}$. Note that $\sigma_{e}(L)$ includes the continuous spectrum of $L$.
\cite{AlexanderGardnerJones}, \cite{Henry},
\cite{Volpert}. Let $C(\mathbb{R})$ be the space of all continuous
functions on the real line and $C_{0}(\mathbb{R})$ be its subspace 

\[
C_{0}(\mathbb{R}):=\{U(\xi)\in C(\mathbb{R})\times C(\mathbb{R})\,|\,\lim_{|\xi|\rightarrow+\infty}U(\xi)=0\}\]

\noindent along with norm\[
\parallel U\parallel_{C_{0}(\mathbb{R})}:=\sup_{\xi\in\mathbb{R}}\parallel U(\xi)\parallel.\]

\noindent We also need the following weighted Banach spaces: for non-negative
numbers $\sigma_{1}$, $\sigma_{2}$, the space $C_{\sigma_{1},\sigma_{2}}$
is defined as:

\[
C_{\sigma_{1},\sigma_{2}}=\{U(\xi)\in C_{0}(\mathbb{R})\mid U(\xi)(e^{\sigma_{1}\xi}+e^{-\sigma_{2}\xi})\in C_{0}(\mathbb{R})\},\]
 on which we define the norm

\[
\parallel U\parallel_{C_{\sigma_{1},\sigma_{2}}}=\sup_{\xi\in\mathbb{R}}\parallel U(\xi)(e^{\sigma_{1}\xi}+e^{-\sigma_{2}\xi})\parallel.\]

\noindent Similarly, we can define $C_{\sigma_{1},\sigma_{2}}^{(i)}$,
$i=1,2,...$ as well, for example:\[
C_{\sigma_{1},\sigma_{2}}^{(2)}=\{U|\,\, U(\xi),U'(\xi),U''(\xi)\in C_{\sigma_{1},\sigma_{;2}}\:\mbox{and\, }\xi\in\mathbb{R}\}\]
 with norm\[
\left\Vert U\right\Vert _{C_{\sigma_{1},\sigma_{2}}^{(2)}}=\sup_{\xi\in\mathbb{R}}\,\,\Sigma_{i=0}^{2}\left\Vert (e^{\sigma_{1}\xi}+e^{-\sigma_{2}\xi})\frac{d^{i}U(\xi)}{d\xi^{i}}\right\Vert .\]
 It can be readily verified that these spaces are Banach spaces. 

\begin{thm}
\textbf{\label{-Theorem-4.1}}Assume [H1] to [H3], and let $U^{*}(\xi)= (u_{1}(\xi),u_{2}(\xi))^{T}$ be the traveling wave solution of (3.6) or (4.1) with wave speed $c^{*}>2\sqrt{\alpha},\,\alpha = 1-\frac{r}{1+\epsilon_{2}}$ as described in Corollary 3.4. Then $U^{*}(\xi)$ is unstable with initial conditions
in $C_{0}$. 
\end{thm}
\begin{proof}
This theorem holds for the traveling wave solutions with critical
and non-critical wave speeds. We need to prove the trivial solution
of (\ref{eq: 4.2}) is unstable. Thus, it suffices to show that in
the space $C_{0}$ the operator $L$ in (\ref{eq: 4.3}) has essential
spectrum with positive real part. As is well known (\cite{Henry},
\cite{Volpert}) the location of the continuous spectrum of the operator
$L$ is bounded by the spectrum of $L$ at $\pm\infty$, which we
denote by $L^{+}$ and $L^{-}$ respectively. More precisely, we let
\begin{equation}
\begin{array}{lll}
L^{+}V & = & V_{\xi\xi}-c^{*}V_{\xi}+{\displaystyle \frac{\partial F}{\partial U}(U_{+}^{*})V}\\
\\ & = & V_{\xi\xi}-c^{*}V_{\xi}+\left[\begin{array}{cc}
-{\displaystyle 1} & r\\
\\0 & \epsilon_{1}-b\end{array}\right]V,\end{array}\label{eq: 4.5}\end{equation}
 \begin{equation}
\begin{array}{lll}
L^{-}V & = & V_{\xi\xi}-c^{*}V_{\xi}+{\displaystyle \frac{\partial F}{\partial U}(U_{-}^{*})V}\\
\\ & = & V_{\xi\xi}-c^{*}V_{\xi}+\left[\begin{array}{cc}
1-\frac{r}{1+\epsilon_{2}} & 0\\
\\\frac{b}{1+\epsilon_{2}} & -\epsilon_{1}\end{array}\right]V.\end{array}\label{eq: 4.6}\end{equation}
 Here, $U_{\pm}^{*}$ respectively denote the limit of $U^{*}(\xi)$
as $\xi\rightarrow\pm\infty$. 

Now consider the equation \[
\frac{\partial V}{\partial t}=L^{+}V.\]
 Following \cite{Volpert} and \cite{Henry}, we replace $V$ by $e^{(\lambda t+i\zeta\xi)}I,$
where $I$ is an identity matrix and $\lambda$ is a complex number
and $\zeta$ is real. We then have \begin{equation}
e^{(\lambda t+i\zeta\xi)}(-\zeta^{2}I-c^{*}\zeta iI+\frac{\partial F}{\partial U}(U_{+}^{*})-\lambda I)=0.\label{eq: 4.7}\end{equation}

The spectrum of the operator $L^{+}$consists of curves given by:
\begin{equation}
\det(-\zeta^{2}I-c^{*}\zeta iI+\frac{\partial F}{\partial U}(U_{+}^{*})-\lambda I)=0.\label{eq: 4.8}\end{equation}
 Solving \eqref{eq: 4.8}, we have \begin{equation}
-\zeta^{2}-c^{*}\zeta i-1-\lambda=0,\label{eq: 4.9}\end{equation}
 or \begin{equation}
-\zeta^{2}-c^{*}\zeta i+\epsilon_{1}-b-\lambda=0.\label{eq: 4.10}\end{equation}
 Letting $\lambda=x+yi$ for $x,y\in\mathbb{R}$, then by (\ref{eq: 4.9})
we have \begin{equation}
x=-\frac{y^{2}}{(c^{*})^{2}}-1,\label{eq: 4.11}\end{equation}
 or by (\ref{eq: 4.10}), \begin{equation}
x=-\frac{y^{2}}{(c^{*})^{2}}+\epsilon_{1}-b.\label{eq: 4.12}\end{equation}

\noindent Similarly, the spectrum of $L^{-}$ consists of
curves: \begin{equation}
x=-\frac{y^{2}}{(c^{*})^{2}}+1-\frac{r}{1+\epsilon_{2}},\label{eq: 4.13}\end{equation}
or \begin{equation}
x=-\frac{y^{2}}{(c^{*})^{2}}-\epsilon_{1}\label{eq: 4.14}\end{equation}
in the complex plane. Consequently, by theory described in \cite{Henry}, we have

\[
\max Re\,\sigma_{e}(L)\ge \max\{-1,\,\epsilon_{1}-b,\,1-\frac{r}{1+\epsilon_{2}},\,-\epsilon_{1}\}=1-\frac{r}{1+\epsilon_{2}}>0.\]

Hence, by \cite{Henry} again, the traveling wave solution $U^{*}(\xi)$
of (\ref{eq: 4.1}) is essentially unstable in $C_{0}(\mathbb{R})$. 
\end{proof}
\vspace{5pt}
 In order to obtain stability for the traveling solution $U^{*}$,
we will restrict the initial conditions and the operator $L$ to a
\char`\"{}smaller\char`\"{} Banach space $C_{\sigma_{1},\sigma_{2}}$
with $\sigma_{1}\geq0$, $\sigma_{2}\geq0$ and $\sigma_{1}^{2}+\sigma_{2}^{2}\neq0$.
To relate the operator $L$ in $C_{0}(\mathbb{R})$ to an
equilvalent operator in $C_{\sigma_{1},\sigma_{2}}$, we introduce
the mapping $T: C_{\sigma_{1}\sigma_{2}} \rightarrow C_{0}$ as follows:
\begin{equation}
TV:=(e^{\sigma_{1}\xi}+e^{-\sigma_{2}\xi})V.\label{eq:4.15}\end{equation}
$T$ is thus linear, bounded and has a bounded inverse $T^{-1}:C_{0}\rightarrow C_{\sigma_{1},\sigma_{2}}$
with $T^{-1}V=(e^{\sigma_{1}\xi}+e^{-\sigma_{2}\xi})^{-1}V$. Consider
operator \begin{equation}
\tilde{L}:V=TLT^{-1}V.\label{eq: 4.16}\end{equation}
One readily sees that $\tilde{L}$ is a linear operator
with domain $C^{(2)}(\mathbb{R})\times C^{(2)}(\mathbb{R})$. By relation
(\ref{eq: 4.16}), considering $L$ in $C_{\sigma_{1},\sigma_{2}}$
is equivalent to considering $\tilde{L}$ in $C_{0}(\mathbb{R})$,
which is: \begin{equation}
\tilde{L}V=V_{\xi\xi}-(2g_{1}+c^{*})V_{\xi}+(2g_{1}^{2}-g_{2}+c^{*}g_{1}+\frac{\partial F}{\partial U}(U^{*}))V,\label{eq: 4.17}\end{equation}
where $(2g_{1}+c^{*})$, $2g_{1}^{2}-g_{2}+c^{*}g_{1}$ in (\ref{eq: 4.17})
are short notions for the matrices $(2g_{1}+c^{*})I$ and $M(\xi)\doteq(2g_{1}^{2}-g_{2}+c^{*}g_{1})I$,
where\[
\begin{array}{lll}
g_{1}(\xi) & = & {\displaystyle \frac{\sigma_{1}e^{\sigma_{1}\xi}-\sigma_{2}e^{-\sigma_{2}\xi}}{e^{\sigma_{1}\xi}+e^{-\sigma_{2}\xi}}},\\
\\g_{2}(\xi) & = & {\displaystyle \frac{\sigma_{1}^{2}e^{\sigma_{1}\xi}+\sigma_{2}^{2}e^{-\sigma_{2}\xi}}{e^{\sigma_{1}\xi}+e^{-\sigma_{2}\xi}}},\end{array}\]
 with \[
\begin{array}{ll}
\lim_{\xi\rightarrow\infty}g_{1}(\xi)=\sigma_{1},\quad & \lim_{\xi\rightarrow-\infty}g_{1}(\xi)=-\sigma_{2};\\
\\\lim_{\xi\rightarrow\infty}g_{2}(\xi)=\sigma_{1}^{2}, & \lim_{\xi\rightarrow-\infty}g_{2}(\xi)=\sigma_{2}^{2}.\end{array}\]
 We now locate the essential spectrum of the operator $\tilde{L}$
in the space $C_{0}(\mathbb{R})$. 

\begin{lem}
\textbf{\label{-Lemma-4.2}}Suppose $\sigma_{1}$ and $\sigma_{2}$
satisfying

\begin{equation}
\begin{array}{c}
0\leq\sigma_{1}<\frac{-c^{*}+\sqrt{c^{*2}+4(b-\epsilon_{1})}}{2},\\
\\0 < \frac{c^{*}-\sqrt{c^{*2}-4(1-\frac{r}{1+\epsilon_{2}})}}{2}<\sigma_{2}<\frac{c^{*}+\sqrt{c^{*2}-4(1-\frac{r}{1+\epsilon_{2}})}}{2};\end{array}\label{eq:4.18}\end{equation}
then the essential spectrum of the operator $\tilde{L}$
in the space $C_{0}(\mathbb{R})$ is contained
in some closed sector in the left half complex plane with vertex on
the horizontal axis left of the origin. Outside this sector, there
are only a finite number of eigenvalues of $\tilde{L}$. 
\end{lem}
\begin{proof}
As in the proof of Theorem \ref{-Theorem-4.1}, we first study the
opearator $\tilde{L}$ at infinity. We have \begin{equation}
\tilde{L}^{+}V=V_{\xi\xi}-(2\sigma_{1}+c^{*})V_{\xi}+(\sigma_{1}^{2}+c^{*}\sigma_{1}+\frac{\partial F}{\partial U}(U^{*}))V,\label{eq: 4.19}\end{equation}
 \begin{equation}
\tilde{L}^{-}V=V_{\xi\xi}-(-2\sigma_{2}+c^{*})V_{\xi}+(\sigma_{2}^{2}-c^{*}\sigma_{2}+\frac{\partial F}{\partial U}(U^{*}))V,\label{eq: 4.20}\end{equation}
 where $\sigma_{1}^{2}+c^{*}\sigma_{1}+\frac{\partial F}{\partial U}(U^{*})$
and $\sigma_{2}^{2}-c^{*}\sigma_{2}+\frac{\partial F}{\partial U}(U^{*})$
correspond respectively to the matrices: \begin{equation}
M^{+}=\left[\begin{array}{cc}
\sigma_{1}^{2}+c^{*}\sigma_{1}-1 & r\\
\\0 & \sigma_{1}^{2}+c^{*}\sigma_{1}+\epsilon_{1}-b\end{array}\right]\label{eq:4.21}\end{equation}
 and \begin{equation}
M^{-}=\left[\begin{array}{cc}
\sigma_{2}^{2}-c^{*}\sigma_{2}+1-\frac{r}{1+\epsilon_{2}} & 0\\
\\\frac{b}{1+\epsilon_{2}} & \sigma_{2}^{2}-c^{*}\sigma_{2}-\epsilon_{1}\end{array}\right].\label{eq:4.22}\end{equation}

Similar to the proof of Theorem \ref{-Theorem-4.1}, we find the right most points of the corresponding parabolas are on the horizontal axis given by

\begin{equation}
\max.\{\sigma_{1}^{2}+c^{*}\sigma_{1}-1,\,\sigma_{1}^{2}+c^{*}\sigma_{1}+\epsilon_{1}-b,\sigma_{2}^{2}-c^{*}\sigma_{2}{\displaystyle +1-\frac{r}{1+\epsilon_{2}},\,\sigma_{2}^{2}-c^{*}\sigma_{2}-\epsilon_{1}\}}.\label{eq: 4.23}\end{equation}
 A simple calculation shows the number above is negative by the choice of $\sigma_{1}$ and $\sigma_{2}$ in (4.18). Thus by the theory in \cite{Volpert}, the essential spectrum of $\tilde{L}$ is contained in a closed sector in the left complex plane with vertex on the horizontal axis left of the origin.
Moreover, we may choose this sector with the further property that outside it there is a finite number of eigenvalues of $\tilde{L}$.
\end{proof}
\begin{cor}
\textbf{\label{cor:4.3}} Assuming all the hypotheses of Lemma \ref{-Lemma-4.2}
and $\sigma_{1},\,\sigma_{2}$ satisfying (\ref{eq:4.18}), the essential
spectrum of the operator $L$ in the space $C_{\sigma_{1},\sigma_{2}}$ is contained
in some closed sector in the left half complex plane with vertex on
the horizontal axis left of the origin. Outside this sector, there
are only a finite number of eigenvalues of $L$. 
\end{cor}
\begin{proof}
The conclusion follows immediately from Lemma \ref{-Lemma-4.2} and
relation (\ref{eq: 4.16}). 
\end{proof}
\medskip{}

Having established the location of the essential spectrum of the operator
$L$ in the space $C_{\sigma_{1},\sigma_{2}}$, we next study the
location of its eigenvalues. We first note that from Corollary 3.4, for $c > 2\sqrt{\alpha}$, we have $(U^{*}(\xi))'(e^{\sigma_{1}\xi}+e^{-\sigma_{2}\xi})$
is unbounded as $\xi\rightarrow-\infty$, which is different from
the situations met in \cite{BatesChen}, \cite{XuZhao}, therefore
their methods can not be carried over to our case. 

\begin{lem}
\textbf{\label{-Lemma-4.4.}} Let $\sigma_{1}$ and $\sigma_{2}$ satisfy (4.18). Then $0$ is not an eigenvalue of the operator $L$ in the space $C_{\sigma_{1}\sigma_{2}}(\mathbb{R})$.
\end{lem}
\begin{proof}
Let $U^{*}$ be a traveling wave solution of (\ref{eq: 4.1}) as obtained
in Theorem \ref{thm: 3.2}. It is easy to see that $(U^{*})'\in C_{0}$
and satisfies the equation \begin{equation}
LV=0.\label{eq: 4.24}\end{equation}
This shows that $0$ is an eigenvalue of the operator $L$
in $C_{0}$. Suppose that there exists a nonzero function $\bar{V} \in C_{\sigma_{1}\sigma_{2}}$ satisfying equation (4.24), we then claim that the ineqaulity $|r\bar{V}(\xi)|\leq(U^{*})'(\xi)$ is consequently true for all $r\in \mathbb{R}$ and all $\xi \in \mathbb{R}$. Writing
$$
S:=\{r\in\mathbb{R}|\,|r\bar{V}(\xi)|\leq(U^{*})'(\xi), \xi \in \mathbb{R}\},
$$
we verify the following properties:

1. $S$ is non-empty, since $0\in S$.

2. $S$ is closed. Let $r_{i}\in S$, $i=1,\,2,\,...$ and $r_{i}\rightarrow r$
as $i\rightarrow+\infty$, then we will have $|r_{i}\bar{V}(\xi)|\leq(U^{*})'(\xi)$
which implies that $|r\bar{V}(\xi)|\leq(U^{*})'(\xi)$, we therefore have
$r\in S$.

3. $S$ is open. Let $r\in S$, we will show that there exists a $\bar{\delta}>0$
such that $(r-\bar{\delta},\, r+\bar{\delta})\subset S$. We claim
that $|r\bar{V}(\xi)|\leq(U^{*})'(\xi)$ implies $|r\bar{V}(\xi)|<(U^{*})'(\xi)$.
In fact, let $W(\xi)=(U^{*})'(\xi)-r\bar{V}(\xi)$ then $W(\xi)\geq0$,
$\xi\in\mathbb{R}$ and satisfies the following equation:\begin{equation}
\left\{ \begin{array}{l}
w_{1}''-cw_{1}'+A_{11}w_{1}+A_{12}w_{2}=0,\\
w_{2}''-cw_{2}'+A_{21}w_{1}+A_{22}w_{2}=0,\\
(w_{1},w_{2})(-\infty)=(w_{1},w_{2})(+\infty)=0,\end{array}\right.\label{eq:4.25}\end{equation}
where $A_{ij},$ $i,j=1,2$ are the entries of the Jacobian $\frac{\partial F}{\partial U}(U^{*})$. Since $A_{12}\ge 0$ and $A_{21}\ge 0$,
the Maximum Principle implies that $W(\xi)=(w_{1}(\xi),w_{2}(\xi))^{T}>0$
for $\xi\in\mathbb{R}$, unless $W$ is identically $0$ and the Lemma is proved. We thus have $(U^{*})'(\xi)-r\bar{V}(\xi)>0,\:\xi\in\mathbb{R}$.
Similarly we can show that $(U^{*})'(\xi)+r\bar{V}(\xi)>0$ for $\xi\in\mathbb{R}$.
The claim then follows.

We next show that the claim further implies $|\bar{r}\bar{V}(\xi)|<(U^{*})'(\xi)$
as long as $\bar{r}$ is sufficiently close to $r$. According to condition (4.18) and the assumption that $\bar{V} \in C_{\sigma_{1}\sigma_{2}}$, for any fixed $\tilde{r} \in \mathbb{R}$, there exists $N>0$ sufficiently large such that $(e^{\sigma_{1}\xi}+e^{-\sigma_{2}\xi})[(U^{*})'(\xi)-\tilde{r}\bar{V}(\xi)]>0$ for all $\xi \in (-\infty, N]$. This implies that $\tilde{r}\bar{V}(\xi)< (U^{*})'(\xi)$ also holds there. Furthermore, due to the claim, and the boundedness of the functions $(U^{*})'$ and $\bar{V}$, we can find $\bar{\delta}>0$
such that for any $\bar{r} \in (-\bar{\delta}+r, \bar{\delta}+r)$,
we have $(U^{*})'(\xi)>\bar{r}\bar{V}(\xi)$ on the finite interval $[-N,N]$.
\par
 Now we fix $\tilde{r}=\bar{r}$ and show $[(U^{*})'(\xi)-\bar{r}\bar{V}(\xi)]>0$ for $\xi \in [N,+\infty)$.
Noting the diagonal entries of the matrix $\frac{\partial F}{\partial U}(U^{*}(+\infty)))$ are both negative, we can choose column vector $P_{+}>0$ such that (increasing $N$ if necessary) $\frac{\partial F}{\partial U}(U^{*}(+\infty)))P_{+}<0$ for $\xi \in [N,+\infty)$.
\par
We have to consider the following two cases:

Case A. If we already have $[(U^{*})'-\bar{r}\bar{V}(\xi)]\geq0$
for $\xi\ge N$, then the Maximum Principle implies that $[(U^{*})'(\xi)-\bar{r}\bar{V}(\xi)]>0$
on $[N,+\infty)$. Analogously, we have $[(U^{*})'(\xi)+\bar{r}\bar{V}(\xi)]>0$
is also true for $\xi\in\mathbb{R}$. Consequently, $S$ is open.

Case B. If there is a point in the interval $(N,+\infty)$ such that
one of the components of vector $(U^{*})'(\xi)-\bar{r}\bar{V}(\xi)$
takes negative local minimum at this point, we consider function $\tilde{W}(\xi):=(U^{*})'(\xi)-\bar{r}\bar{V}(\xi)+\tau P_{+}$.
The asymptotic rates of $(U^{*})'$ and $\bar{V}$ at $+\infty$ imply
that there is a sufficiently large $\tau>0$ such that $\tilde{W}(\xi)= (U^{*})'(\xi)-\bar{r}\bar{V}(\xi)+\tau P_{+}\geq0$
for $\xi\in(N,+\infty)$. We further assume that one of the components
of $\tilde{W}(\xi)$, say $\tilde{w_{1}}$ for example, takes minimum
at a finite point $\xi_{2}\,\mbox{in}\,(N,+\infty)$. It is not hard
to verify that there is a $\tau=\tau_{2}$ such that the corresponding
$\tilde{W}(\xi)$ staisfying $\tilde{w_{1}}(\xi_{2})=0$ and $\tilde{W}(\xi)\geq0$
for $\xi\in(N,+\infty)$. For such $\tau_{2}$ on the one hand, we
have\begin{equation}
\begin{array}{lll}
L\tilde{W} & = & \tilde{W}_{\xi\xi}-c^{*}\tilde{W}_{\xi}+\frac{\partial F}{\partial U}(U^{*})\tilde{W}\\
 & = & \tau_{2}\frac{\partial F}{\partial U}(U^{*}))P_{+}<0;\end{array}\label{eq:4.26}\end{equation}
on the other hand at $\xi=\xi_{2}$, the first component on the left hand side of (\ref{eq:4.26})
is larger than or equal to zero. We then have a contradiction, and consequently
$(U^{*})'(\xi)-\bar{r}\bar{V}(\xi)\geq0$ for $\xi\in(N,+\infty)$.
We are again in the situation descibed by case A. By a similar argument, we can show that $(U^{*})'(\xi)+\bar{r}\bar{V}(\xi) \ge 0$ for $\xi \in [N,\infty)$.
\par
In summary, both case A and Case B show that for any $\bar{r}\in(-\bar{\delta}+r,\,\bar{\delta}+r)$,
$|\bar{r}\bar{V}(\xi)|<(U^{*})'(\xi), \xi \in \mathbb{R}$, i.e., $S$ is open. 

Now the set $S$ is a non-empty, open and closed subset of $\mathbb{R}$,
hence $S\equiv\mathbb{R}$. However, this is impossible by the definition
of $S$, since $(U^{*})'(\xi)$ is bounded. Therefore the equation $LV=0$ cannot have a nontrivial solution in $C_{\sigma_{1}\sigma_{2}}$.

\end{proof}

The next lemma shows that there is no eigenvalue of the operator $L$
in $C_{\sigma_{1},\sigma_{2}}$ with positive real part. 

\begin{lem}
\label{lem:4.5} Let $C_{0}^{\mathbb{C}}$ be the complexified space
of $C_{0}(\mathbb{R})$ and $\lambda$ be an eigenvalue of the operator
$\tilde{\mathcal{L}}$, given by (4.16), with corresponding eigenfunction $\underline{U}\in C_{0}^{\mathbb{C}}$,
then $\mbox{Re}\:\lambda<0$.
\end{lem}
\begin{proof}
Let the eigenvalue $\lambda=\lambda_{1}+\lambda_{2}i$ and eigenfunction $\underline{U}(\xi)=U^{1}(\xi)+iU^{2}(\xi)$
for $\xi\in\mathbb{R}$, where $\lambda_{i}\in\mathbb{R}$ and $U^{i}(\xi)\in C_{0}(\mathbb{R})$.

Consider the Cauchy problem (\cite{BatesChen}, \cite{XuZhao}): 

\begin{equation}
V_{t}=\tilde{\mathcal{L}}V-\lambda_{1}V,\quad V(\xi,0)=U^{1}(\xi).\label{eq:4.27}\end{equation}
It is easy to verify that $V(\xi,t)=U^{1}(\xi)\cos(\lambda_{2}t)-U^{2}(\xi)\sin(\lambda_{2}t)$
solves \eqref{eq:4.27} for $\xi\in\mathbb{R}$ and $t\geq0$ and
is bounded. We suppose that at least one of the components of $V$ assumes
positive value at some $\xi$ and $t$ (we can consider $-V$ if otherwise).
\par
Suppose $\lambda_{1}\ge 0$, then the following claim is true.

\textit{Claim}: \textit{There exists a $r>0$ such that $V(\xi,t)\leq r \mathcal{T}(U^{*})'(\xi)$ 
for $\xi\in\mathbb{R}$ and $t\geq0$.} (Recall the operator $\mathcal{T}$ is defined in (4.15).)

In fact since the vector $\mathcal{T}(U^{*})'(\xi)\rightarrow+\infty$
as $\xi\rightarrow-\infty$ we can choose a sufficiently large $\xi_{0}>0$
such that 

\begin{equation}
V(\xi,t)<\mathcal{T}(U^{*})'(\xi)\;\mbox{for}\:\xi\leq-\xi_{0}\:\mbox{and}\: t\geq0.\label{eq:4.28}\end{equation}
Furthermore the positivity of $\mathcal{T}(U^{*})'(\xi)$ for $\xi\in\mathbb{R}$
implies that there is a $r>0$ such that 

\begin{equation}
V(\xi,t)\leq r\mathcal{T}(U^{*})'(\xi)\quad\mbox{for}\:\xi\in[-\xi_{0},\xi_{0}]\:\mbox{and}\: t\geq0.\label{eq:4.29}\end{equation}

Let $\bar{r}=\max\{1,r\}$, we then have 

\begin{equation}
V(\xi,t)\leq\bar{r}\mathcal{T}(U^{*})'(\xi)\label{eq:4.30}\end{equation}
for $\xi\leq\xi_{0}$ and $t\geq0$. We next adjust $\bar{r}$ suitably 
such that an equality in \eqref{eq:4.30} holds on at least one component
at a point $(\xi_{1},t_{1})$ with $\xi_{1}\in(-\infty,\xi_{0}]$
and $t_{1}\geq0$. 

We proceed to show that the assumption $\lambda_{1} \ge 0$ implies that \eqref{eq:4.30} is also true
for $\xi\geq\xi_{0}$ and $t\geq0.$

From the limits of $g_{1}, g_{2}$, the choice of $\sigma_{1}, \sigma_{2}$, and (4.21), we can find a vector $\tilde{P}^{+}>0$ and increasing $\xi_{0}$
if necessary such that
\begin{equation}
M(\xi)\tilde{P}^{+}<0\quad\mbox{for}\:\xi\geq\xi_{0},\label{eq:4.31}\end{equation}
where
$$
M(\xi):= (2g_{1}^{2}(\xi)-g_{2}(\xi)+ c^{*}g_{1}(\xi))I + \frac{\partial F}{\partial U}(U^{*}(\xi)).
$$
 We can also choose a small $\bar{\epsilon}>0$ such that 

\begin{equation}
(\bar{\epsilon}^{2}\begin{array}{c}
\left(\begin{array}{c}
1\\
1\end{array}\right)\end{array}-\bar{\epsilon}(2g_{1}+c^{*})I+M(\xi))\tilde{P}^{+}<0\quad\mbox{for}\:\xi\geq\xi_{0},\label{eq:4.32}\end{equation}

Now suppose that we can find a $\xi_{1}>\xi_{0}$ and a $t_{1}\geq0$
such that $V(\xi_{1},t_{1})>\bar{r}\mathcal{T}(U^{*})'(\xi_{1})$.
Let
$$
Q^{+}(\xi):=e^{\bar{\epsilon}\xi}\tilde{P}^{+}.
$$
Since $Q^{+}(\xi)\rightarrow+\infty$ as
$\xi\rightarrow+\infty$, there is a $\tilde{r}>0$ such that $V(\xi,t)\leq\bar{r}\mathcal{T}(U^{*})'(\xi)+\tilde{r}Q^{+}(\xi)$
for all $\xi\geq\xi_{0}$ and $t\geq0$, and for at least one index
$j$ and a $\xi_{2}\geq\xi_{0}$ and a $t_{2}>0$, we have the equality
for the $j-th$ component:

\[
V_{j}(\xi_{2},t_{2})=\bar{r}\mathcal{T}(U_{j}^{*})'(\xi_{2})+\tilde{r}Q_{j}^{+}(\xi_{2}).\]
Let $Y(\xi,t)=\bar{r}\mathcal{T}(U^{*})'(\xi)+\tilde{r}Q^{+}(\xi)-V(\xi,t)$\textbf{,
}then $Y_{j}$ has the following properties: $Y_{j}(\xi_{2},t_{2})=0,$
$Y_{j}(\xi_{0},t)>0$, $Y_{j}(\xi,t)\geq0$ for all $\xi\geq\xi_{0}$,
$t\geq0$ and it then follows that $Y_{j,t}(\xi_{2},t_{2})= 0$,
$Y_{j,\xi}(\xi_{2},t_{2})=0$ and $Y_{j,\xi\xi}(\xi_{2},t_{2})\geq0$
and that $Y_{j}(\xi,t)$ satisfies

\begin{equation}
\begin{array}{lll}
Y_{j,t} & = & -V_{j,t}\\
 & = & -(\tilde{\mathcal{L}}V-\lambda_{1}V)_{j}\\
 & > & (-\tilde{\mathcal{L}}V+\lambda_{1}V+\tilde{\mathcal{L}}\bar{r}\mathcal{T}(U^{*})'+\tilde{\mathcal{L}}\tilde{r}Q^{+}-\lambda_{1}(\bar{r}\mathcal{T}(U^{*})'+\tilde{r}Q^{+}))_{j}\\
 & = & (\tilde{\mathcal{L}Y}-\lambda_{1}Y)_{j}\\
 & = & Y_{j,\xi\xi}-(2g{1}+c^{*})Y_{j,\xi}+M_{1j}(U^{*})Y_{1}+M_{2j}(U^{*})Y_{2}-\lambda_{1}Y_{j}.\end{array}\label{eq:4.33}\end{equation}
Note that in the third line above we have $\tilde{\mathcal{L}}\bar{r}\mathcal{T}(U^{*})'=0$,\,$\tilde{\mathcal{L}}\tilde{r}Q^{+}<0$ by (4.32),\,$\lambda_{1}\ge 0$, $\bar{r}\mathcal{T}(U^{*})'>0$ and $\tilde{r}Q^{+}>0$. In the last line of (4.33), $M_{ij}$ denotes the ij-th entry of the matrix $M$ in (4.31).
However, at $(\xi_{2},t_{2})$ the left hand side of \eqref{eq:4.33}
is equal to zero, while the right hand side is greater than or equal to zero because the off diagonal entries of $M$ is nonnegative. We have a contradiction and consequently the claim is proved.

We have thus established the fact that the set
$$
S=\{r\geq0|V(\xi,t)\leq r\mathcal{T}(U^{*})'(\xi)\;\mbox{for}\;\xi\in\mathbb{R}\}
$$
is non-empty. Let $r_{0}$ denotes the greatest lower bound of $S$. In what follows we will
show $r_{0}=0$. Suppose $r_{0}>0$, we have \begin{equation}
V(\xi,t)\leq r_{0}\mathcal{T}(U^{*})'(\xi)\;\mbox{for}\;\xi\in\mathbb{R}.\label{eq:4.34}\end{equation}
We first assume that an equality occurs at a point in \eqref{eq:4.34}
at the $i-th$ component at a point $(\tilde{\xi},\tilde{t})$ with
$\tilde{\xi}\in\mathbb{R}$ and $\tilde{t}\geq0$. Let
$$
X(\xi,t):=r_{0}\mathcal{T}(U^{*})'(\xi)-V(\xi,t).
$$
From (4.27), we obtain the follwing inequality

\[
\begin{array}{lll}
X_{i,t} & \geq & (\tilde{\mathcal{L}}X-\lambda_{1}X)_{i}\\
 & = & X_{i,\xi\xi}-(2g_{i}+c^{*})X_{i,\xi}+M_{1i}(U^{*})X_{1}+M_{2i}(U^{*})X_{2}-\lambda_{1}X_{i}\\
 & \ge & X_{i,\xi\xi} -(2g_{i}+c^{*})X_{i,\xi}+M_{ii}(U^{*})X_{i}-\lambda_{1}X_{i}
\end{array}\]
The last inequality is true because $M_{ij}(U^{*})\ge 0$ if $i\neq j$.
 From the
positivity theorem for parabolic equations we deduce that $X_{i}(\xi,t)>0$
for $\xi\in\mathbb{R}$ and $t>\tilde{t}$ (cf p.14 in \cite{Leung}). However by the $t-$periodicity
of $V$, we have that $X_{i}(\xi,t)>0$ for all $\xi\in\mathbb{R}$
and $t>0$. Contradiction with the existence of $\tilde{\xi}$. This
shows that 

\begin{equation}
V(\xi,t)<r_{0}\mathcal{T}(U^{*})'(\xi)\;\mbox{for}\;\xi\in\mathbb{R},\: t\geq0.\label{eq:4.35}\end{equation}

Again since $\mathcal{T}(U^{*})'(\xi)\rightarrow+\infty$ monotonically
as $\xi\rightarrow-\infty$, there exist a sufficiently small $\delta_{1}>0$
and a large $\bar{M}>0$ such that \begin{equation}
V(\xi,t)<(r_{0}-\delta_{1})\mathcal{T}(U^{*})'(\xi)\quad\mbox{for }\xi\leq-\bar{M}\;\mbox{and}\; t\geq0.\label{eq:4.36}\end{equation}
The positivity of $\mathcal{T}(U^{*})'(\xi)$ implies that we can
extend inequality \eqref{eq:4.36} to the interval $(-\infty,\bar{M})$
with an even smaller $\delta>0$. We have 

\[
V(\xi,t)<(r_{0}-\delta)\mathcal{T}(U^{*})'(\xi)\;\mbox{for}\;\xi\in(-\infty,\bar{M)\;}\mbox{and}\; t\geq0.\]

We are then in the same situation as in the proof of the claim at the beginning of this lemma. Using similar arguments as in the proof of the claim, we extend the above inequality to: 

\[
V(\xi,t)\ge (r_{0}-\delta)\mathcal{T}(U^{*})'(\xi)\;\mbox{for}\;\xi\in\mathbb{R},\: t>0.\]
It then follows that $r_{0}-\delta\in S.$ Contradiction with the
definition of $r_{0}$. Hence $r_{0}=0$. However, this contradicts the assumption that at least one component of $V$ assume positive value. Thus we must have $\lambda_{1} <0$. This concludes the proof of the lemma. 
\end{proof}
\begin{thm}
\label{thm:4.7}Assume [H1] to [H3] and that $\sigma_{1}$ and $\sigma_{2}$ satisfy
\eqref{eq:4.18}, the operator $L$ in $C_{\sigma_{1},\sigma_{2}}$
has a dense domain of definition. For any complext number with $Re\,\lambda>0$ large enough,
$(\lambda-L)^{-1}$exists and is defined on all of $C_{\sigma_{1},\sigma_{2}}$,
and satisfies the following estimate \begin{equation}
\left\Vert (\lambda-L)^{-1}\right\Vert _{C_{\sigma_{1},\sigma_{2}}}\leq\frac{\bar{c}}{1+|\lambda|},\label{eq:4.37}\end{equation}
 where $\bar{c}>0$ is a constant. 
\end{thm}
\noindent \vspace{5pt}

\begin{proof}
The proof follows the same idea as in \cite{Volpert} but with resolvent
estimates in $C_{0,\tau}$ replaced by in the space $C_{\sigma_{1},\sigma_{2}}$.
We skip the proof. 
\end{proof}
\begin{thm}
\label{thm:4.7}Under the hypotheses of Theorem \ref{thm:4.7} the
operator $L$ generates an analytical semigroup in $C_{\sigma_{1},\sigma_{2}}$,
where $\sigma_{1}$ and $\sigma_{2}$ satisfy \eqref{eq:4.18}. 
\end{thm}
\begin{proof}
The conclusion follows from Theorem \eqref{thm:4.7}and Hille-Yoshida
Theorem. 
\end{proof}
We now state the stability theorem, 

\begin{thm}
\label{thm:4.8}Assume hypotheses [H1] to [H3], and $\sigma_{1}, \sigma_{2}$ satisfy (4.18). The
traveling wave solution $U^{*}$ of \eqref{eq:3.6}-\eqref{eq:3.7},
with wave speed $c^{*}>2\sqrt{\alpha}$ , is asymptotically stable
according to norm $||\cdot||:=||\cdot||_{C_{\sigma_{1},\sigma_{2}}}$.
That is, there exists $\epsilon>0$ such that if the initial condition
$U(\xi,0)=\bar{U}(\xi)\in C$ with $(\bar{U}(\xi)-U^{*}(\xi))\in C_{\sigma_{1},\sigma_{2}}$
and $||\bar{U}-U^{*}||<\epsilon$, then the solution $U(\xi,t)$ exists
uniquely for all $t>0$ and  \begin{equation}
||U(\xi,t)-U^{*}(x+ct)||\leq Me^{-bt}.\label{eq:4.43}\end{equation}
Here, the constants $M>0,\, b>0$ are independent of $t$ and $\bar{U}$. 
\end{thm}
\begin{proof}
The stability of $U^{*}$ leads to the consideration of the stability of the trivial solution for system (4.2), and the analysis of the spectrum of the operator $L$ in (4.3). Corollary \ref{cor:4.3}, Lemma 4.4 and Lemma \ref{lem:4.5}
show that the spectrum of the operator $L$ in the space $C_{\sigma_{1},\sigma_{2}}$
is contained in a closed angular region in the left open complex plane. Thus, following the methods in Theorem 2.1 on p.227 in \cite{Volpert}, we obtain the conclusion of this theorem.
\end{proof}

\end{document}